\documentclass[10pt]{amsart}
\usepackage{latexsym,color,amsmath,amsthm,amssymb,amscd,amsfonts}

\newcommand{\R}{\mathbb R}
\newcommand{\K}{\mathcal K}
\newcommand{\F}{\mathcal F}
\newcommand{\N}{\mathbb N}
\renewcommand{\L}{\mathcal L}

\newcommand{\conv}{\mathrm{conv}}

\newcommand{\alfa}{\gamma}

\newcommand{\supp}{\mathrm{supp}}
\newcommand{\inter}{\mathrm{int}}
\newcommand{\lin}{\mathrm{span}}

\def\vol{{\rm vol}}
\def\epi{{\rm epi}}

\newcommand{\iprod}[2]{\langle #1,#2 \rangle} 

\newtheorem{thm}{Theorem}[section]
\newtheorem{lemma}[thm]{Lemma}
\newtheorem{proposition}[thm]{Proposition}
\newtheorem{cor}[thm]{Corollary}

\newtheorem{df}{Definition}[section]
\newtheorem{rmk}[thm]{Remark}

\begin{document}

\title[Rogers-Shephard and Loomis-Whitney inequalities]{Rogers-Shephard and local Loomis-Whitney type inequalities}

\author[D.~Alonso-Guti\'errez]{David Alonso-Guti\'errez}
\address[D.~Alonso]{Universidad de Zaragoza (Spain)}
\email{alonsod@unizar.es}

\author[S.~Artstein-Avidan]{Shiri Artstein-Avidan}
\address[S.~Artstein-Avidan]{Tel-Aviv University (Israel) }
\email{shiri@post.tau.ac.il}

\author[B.~Gonz\'alez]{Bernardo Gonz\'{a}lez Merino}
\address[B. Gonz\'alez]{University Centre of Defence at the Spanish Air Force Academy, MDE-UPCT, Murcia (Spain)}
\email{bernardo.gonzalez@cud.upct.es}

\author[C.H.~Jim\'enez]{C. Hugo Jim\'enez}
\address[C.~H.~Jim\'enez]{Pontif\'icia Universidade Cat\'olica do Rio de Janeiro (Brazil)}
\email{hugojimenez@mat.puc-rio.br}

\author[R.~Villa]{Rafael Villa}
\address[R.~Villa]{Universidad de Sevilla (Spain)}
\email{villa@us.es}

\thanks{The first named author is partially supported by Spanish grant MTM2016-77710-P
projects and by IUMA. The second named author is partially supported by ISF grant number 665/15. The third named author is partially supported by Fundaci\'{o}n S\'{e}neca, Science and Technology Agency of the Regi\'{o}n de Murcia, through the
Programa de Formaci\'on Postdoctoral de Personal Investigador, project reference 19769/PD/15, and
the Programme in Support of Excellence Groups of the Regi\'{o}n de Murcia, Spain, project reference 19901/GERM/15.
The fourth named author is supported by CNPq and the program Incentivo \`a produtividade em ensino e pesquisa from the PUC-Rio.
The fifth named author is partially supported by MINECO project reference MTM2015-63699-P, Spain.}

\date{\today}\maketitle

\begin{abstract}
We provide functional analogues of the classical geometric inequality of Rogers and Shephard on products of volumes of sections and projections. As a consequence we recover (and obtain some new) functional versions of Rogers-Shephard type inequalities as well as some  generalizations of the geometric Rogers-Shephard inequality in the case where the subspaces intersect. These generalizations can be regarded as sharp local reverse Loomis-Whitney inequalities. We also obtain a sharp local Loomis-Whitney inequality.
\end{abstract}

\section{Introduction and main results}

The comparison of the volume of a convex body with the volumes of its
sections and projections can be useful in many situations. Fubini's theorem implies the trivial bound for the volume of an $n$-dimensional convex body $K$
\begin{equation}\label{eq:upperbound}
\vol_n(K) \le \vol_{i}(P_HK) \max_{x_0 \in \R^n}\vol_{n-i}(K\cap(x_0+H^{\bot})),
\end{equation}
where $H$ denotes any $i$-dimensional linear subspace, $P_HK$ the orthogonal projection of $K$ onto $H$, and $\vol_n(K)$ the $n$-dimensional volume of an $n$-dimensional convex body. The reverse bound is a classical theorem of Rogers and Shephard
 \cite[Theorem 1]{RS58}: For any convex body $K\subset \R^n$ and any $i$-dimensional subspace $H$
 we have
\begin{equation}\label{eq:volumesplit}
 \vol_{i}(P_HK)  \vol_{n-i}(K\cap H^{\bot}) \le  {n\choose i}\vol_n(K).
\end{equation}
They also showed that equality holds if and only if
for every $v\in H$ the intersection $K\cap (H^{\bot}+\R^+v)$ is the convex hull of $K\cap H^{\bot}$ and one point.

The first goal of this paper is to give a functional analogue of inequality \eqref{eq:volumesplit} for log-concave functions. Log-concave measures naturally arise in Convex Geometry, firstly because the Brunn-Minkowski inequality establishes the log-concavity of the Lebesgue measure and of the marginals of the uniform measure on convex sets and, secondly, because the class of log-concave functions is the smallest class, closed under limits, that contains the densities of such marginals.

The space of log-concave measures
has shown to be fundamental in several areas of mathematics. From a functional point of view they resemble Gaussian
functions in many different ways. Many functional inequalities satisfied by Gaussian functions, like Poincar\'e and Log-Sobolev inequalities, also hold in a more general subclass of log-concave functions \cite{BBCG,Bob}. They also appear in areas like Information Theory, in the study of some important parameters, such as the classical entropy \cite{BM1}. Besides, there are in the literature many examples of functional inequalities with a geometric counterpart;
Pr\'ekopa-Leindler/Brunn-Minkowski \cite{Prek} and Sobolev/Petty projection \cite{Zha} inequalities are two of the main
examples. This has generated an increasing interest in extending several important parameters of convex bodies to functional parameters \cite{AGJV,AGJV2,AKM,AKSW,BCF,BM2,Col,CLM,FrMe,KM}.

Moving back to the geometric world, one may ask for volume comparisons in the sense of \eqref{eq:upperbound} and \eqref{eq:volumesplit} when
the two given subspaces present a non-trivial intersection. These questions have been repeatedly addressed several times in the last decade \cite{CaGr,BoTh,BGL,FrGiMe,GiHaPa,SoZv,Xiao}. In essence, we will show these new type of inequalities using different tools.
In order to derive them, we will use, on the one hand, the functional extensions of \eqref{eq:volumesplit} announced above,
and on the other hand, Berwald's second
inequality (cf.~Appendix in Section \ref{sec:appendixBerwald}) to provide extensions of inequalities of the type of \eqref{eq:upperbound}.
These approaches show that the log-concave measure settings present the right level of complexity for these type of questions.

A function $f:\R^n\rightarrow[0,\infty)$ is called log-concave if there exists a convex function $u:\R^n\rightarrow(-\infty,\infty]$ such that
$f=\exp(-u)$, or equivalently, if
\[
f((1-\lambda)x+\lambda y)\geq f(x)^{1-\lambda}f(y)^\lambda,
\]
for every $x,y\in\R^n$ and $\lambda\in[0,1]$. For a convex body $K\in\mathcal K^n$ we denote by $\chi_K$ its characteristic function, i.e.,
\[
\chi_K(x):=\left\{\begin{array}{cc}1 & \text{if }x\in K\\ 0 & \text{ otherwise.}\end{array}\right.
\]
We denote by $\F(\R^n)$ the set of log-concave integrable functions on $\R^n$.
By $\L^n_i$ we denote the set of all $i$-dimensional linear subspaces of $\R^n$.
Given $f\in\mathcal{F}(\R^n)$ and $H\in\mathcal{L}^n_i$, the projection of $f$ onto $H$ (also called the ``shadow'' of $f$, not to be confused with its marginal, cf.~\cite[Pg.~178]{KM}) is defined by
$$
P_Hf(x):=\max\{f(y)\,:\,y\in x+H^\perp\}\quad\forall x\in H.
$$
We show
\begin{thm}\label{th:splittRS}
	Let $f\in\mathcal{F}(\R^n)$ and $H\in\mathcal{L}^n_i$. Then 
	\begin{equation}\label{eq:splittingTheorem}
\int_{H}P_Hf(x)dx \int_{H^{\bot}}f(y)dy	\le \binom{n}{i} \Vert f\Vert_{\infty}\int_{\R^n}f(z)dz.
	\end{equation}
Equality holds if and only if $\frac{f}{\Vert f\Vert_\infty}=\chi_K$, for some  $K\in\K^n$, such that equality holds in \eqref{eq:volumesplit}, i.e., for every
$v\in H$ the intersection $K\cap (H^{\bot}+\R^+v)$ is the convex hull of $K\cap H^{\bot}$ and one point.
\end{thm}

If $f=\chi_K$ is the characteristic function of some $K\in\mathcal K^n$, then \eqref{eq:splittingTheorem} recovers \eqref{eq:volumesplit}.
Interestingly, there is a ``non-linear'' extension of \eqref{eq:volumesplit} which demonstrates a very different facet of the classical inequality, invisible on the purely geometric level. We prove:

\begin{thm}\label{SplittingColesanti}
	Let $f\in\F(\R^n)$ such that $f(0)=\Vert f\Vert_{\infty}$, $H\in\L^n_i$, and $\lambda\in[0,1]$. Then
	\[
	(1-\lambda)^i\lambda^{n-i}\int_{H}P_Hf(x)^{1-\lambda}dx\int_{H^{\bot}}f(y)^{\lambda}dy\leq \int_{\R^n}f(z)dz.
	\]
Equality holds if and only if for every $(x,y)\in H\times H^\perp$, $f(x,y)=\exp(-\Vert (x,y)\Vert_K)$
for some $K\in\mathcal{K}^n$ with $0\in K$ such that for every $v\in H$ the intersection $K\cap (H^{\bot}+\R^+v)$ is the convex hull of $K\cap H^{\bot}$ and one point.
\end{thm}
To see that Theorem \ref{SplittingColesanti} is indeed an extension of \eqref{eq:volumesplit}, plug in
$f(x) = \exp(-\Vert \cdot\Vert_K)$   for some $K\in\mathcal K^n$ with $0\in K$, inequality \eqref{eq:volumesplit} is recovered.
\begin{rmk}
Notice that we allow 0 to be in the boundary of $K$. In such case we understand that
\[
\|x\|_K:=\left\{\begin{array}{lr}\inf\{\rho\geq 0:x\in\rho K\} & \text{if }x\in\rho K\text{ for some }\rho\geq 0\\
\infty & \text{otherwise.}\end{array}\right.
\]
\end{rmk}

Analogously to the extensions above, we also provide a non-linear functional inequality for log-concave functions
in the spirit of \eqref{eq:upperbound}.

\begin{thm}\label{th:upperboundcharact}
Let $f\in\mathcal{F}(\R^n)$, $H\in\mathcal L^n_i$, and $\lambda\in[0,1]$. Then
\[
\int_{\R^n}f(z)dz\leq\int_HP_Hf(x)^{1-\lambda}dx\max_{x_0\in H}\int_{x_0+H^\bot}f(y)^\lambda dy.
\]
\end{thm}

Replacing $f$ by the characteristic function of a convex set recovers \eqref{eq:upperbound}.
Again, also here we have another very different functional version, closer in spirit to the original geometric inequality, but with the section replaced by a projection:

\begin{thm}\label{th:upperboundexponent}
Let $f\in\mathcal F(\R^n)$, $H\in\mathcal L^n_i$. 
Then
\[
\Vert f\Vert_\infty\int_{\R^n}f(z)dz\leq{n\choose i}\int_{H}P_Hf(x)dx\int_{H^\bot}P_{H^\bot}f(y)dy.
\]Equality holds if and only if $\frac{f(z)}{\Vert f\Vert_\infty}=\exp(-\Vert z-z_0\Vert_{K\times L})$ for some $z_0\in\R^n$ and some convex bodies $K\subseteq H$ and $L\subseteq H^\perp$ such that $0\in K\times L$.
\end{thm}


Several recent publications take care of translating classical results of convex geometry to log-concave functions (cf.~\cite{Col2,MMX}).
Indeed, most of our functional results here extend their geometrical counterparts through the natural injections
\[
\begin{array}{ccc}\mathcal K^n & \hookrightarrow & \mathcal F(\R^n) \\
K & \rightsquigarrow & \left\{\begin{array}{c}\chi_K\\\exp(-\Vert\cdot\Vert_K).\end{array}\right.
\end{array}
\]

Within the light of these injections, given $f=\exp(-u)\in\mathcal F(\R^n)$ and $H\in\mathcal L^n_i$,
the projection $P_Hf$ is driven from the orthogonal projection applied to the epi-graph of the convex function $u$. Namely, if $\{e_1,\dots,e_{n+1}\}$ are the canonical vectors of $\R^{n+1}$, and we let
${\rm epi}(u) = \{(x,t)\in\R^{n+1}: u(x)\leq t\}$ and $\overline{H}:=\lin\{H,e_{n+1}\}$, we have that
\[  {\rm epi}(\tilde{u})  = P_{\overline{H}} ( {\rm epi}(u) ) \subset H\times \R, \]
for some $\tilde{u}:H\to\R$ convex, then we have that $P_H f = \exp (-\tilde{u})$.

Therefore, the functional inequalities above are actually inequalities about projections and sections of a convex epi-graph, where the measure is not the usual volume (this would be infinite since the epi-graph is unbounded) but with respect to the weight function $\exp (-t)$ where $t = \iprod{x}{e_{n+1}}$ is the last coordinate. What is more important, is that the projection and the section now have a common 1-dimensional subspace, which is $\{0\}^n\times\R$.

This created a new situation, and  led us to investigate what happens to the volumetric information in the original geometric inequalities \eqref{eq:upperbound} and \eqref{eq:volumesplit}, in the more general case where the subspaces intersect (orthogonally). More precisely, we consider
 sections and
projections of $K$ for two intersecting subspaces of the form
$E\in\mathcal L^n_i$, $H\in\mathcal L^n_j$, $i+j\geq n+1$, $E^\bot\subseteq H$.

These kind of questions give rise to inequalities which somewhat resemble the classical Loomis Whitney inequality. In some literature they are called
``local Loomis-Whitney type inequalities'' (cf.~\cite[Pg.~2]{BGL}). In this regard, we show the following sharp inequalities.

\begin{thm}\label{th:localLoomisWhitneyReverse}
	Let $K\in\K^n$, $E\in\L^n_i$, $H\in\L^n_j$ be such that $i,j\in\{2,\dots,n-1\}$, $i+j\geq n+1$, and $E^{\bot}\subset H$.
    Denote $k = i+j-n$, so that $1\le k\le n-2$.
	Then for any $z\in \R^n$
	\[
	\vol_i(P_EK) \vol_j(K\cap (z+H)) \le {n-k\choose n-i}  \vol_n(K) \max_{x\in \R^n} \vol_{k} (P_E(K\cap(x+H)))    .
	\]
	Moreover, equality holds if and only if  there exist two convex bodies $K_1\subseteq E\cap H$, $K_2\subseteq E^\bot$ such that $S_E(K\cap H)= K_1+ K_2$ and for every $v\in H^\bot$, $K\cap(H+\R^+v)$ is the convex hull of $K\cap H$ and one point. 
Here $S_E$ denotes the symmetral  with respect to the subspace $E$, and is formally defined in Section 2.
\end{thm}

\begin{thm}\label{th:localLoomisWhitney}
	Let $K\in\mathcal{K}^n$, $E\in\mathcal{L}^n_i$, $H\in\mathcal{L}^n_j$ be such that $i,j\in\{2,\dots,n-1\}$, $i+j\geq n+1$,
and $E^{\bot}\subset H$. Denote $k = i+j-n$, so that $1\leq k\leq n-2$. Then
	\begin{equation}\label{dirLocLW}
	\vol_{k}(P_{E\cap H}K) \vol_n(K)\leq\frac{{i\choose k}{j\choose k}}{{n\choose k}}\vol_{i}(P_EK)\vol_{j}(P_HK).
	\end{equation}
	Equality holds if and only if there exist $K_1\subseteq H^\bot$, $K_2\subseteq E^\bot$, $x_0\in P_{E\cap H}K$ such that
	for every $v\in E\cap H$, $K\cap(x_0+(E\cap H)^\bot+\R^+v)$ is the convex hull of $x_0+(K_1+K_2)$ and a unique point.
\end{thm}

In Section \ref{sec:LW} we recall the classical Loomis-Whitney type inequalities, and the known ``local'' results from
\cite{BGL} \cite{GiHaPa} and also \cite{SoZv}, and compare them with our results above. Here we only mention that, on the one hand,
Theorem \ref{th:localLoomisWhitneyReverse} improves  upon \cite[Theorem 1.3]{BGL}, in the particular case of two subspaces.
On the other hand, the case $i = j$ of Theorem \ref{th:localLoomisWhitney} is included in
\cite[Theorem 1.2]{BGL}, and  for $i\neq j$  the constant in our result is slightly better (and sharp). We borrow the idea of using Berwald's inequality in the proof; a version of Theorem \ref{RSwithIntersection} with non-sharp constants appeared in \cite{GiHaPa}.
The novelty here regarding Berwald's inequality lies in the fact that we use, unlike the other authors, \cite[Satz 8]{Ber}. It seems that the literature does not cover a proof in English of that result. For the sake of completeness, we will include the proof of the result we need in the Appendix in Section \ref{sec:appendixBerwald}.
In this direction, we would also like to point out that there exists a reverse Berwald's inequality (cf.~Corollary 3 in \cite{BN}).

As explained above, moving to the functional realm allows us to extend our understanding and gain new insight into the original geometric notions. It is therefore natural to extend further, and find the
 functional analogues of Theorems \ref{th:localLoomisWhitneyReverse} and \ref{th:localLoomisWhitney}. Indeed, we are able to do this, and we prove the following two theorems to this effect.

\begin{thm}\label{th:locLWFunctProjSect}
Let $f\in\mathcal{F}(\R^n)$ and let $H\in\mathcal{L}^n_i$ and $E\in\mathcal{L}^n_j$ be such that $i,j\in\{2,\dots,n-1\}$, $i+j\geq n+1$, and $H^\bot\subset E$.  Denote $k = i+j-n$, so that $1\le k\le n-2$. Then
\[
\int_EP_Ef(x)dx\,\int_Hf(y)dy \leq {n-k\choose n-i}\max_{x_0\in H^\bot}\int_{x_0+E\cap H}P_Ef(w)dw\,\int_{\R^n}f(z)dz.
\]
\end{thm}

\begin{thm}\label{thm: FunctionalLW}
Let $f\in\F(\R^n)$ and let $E\in\mathcal{L}^n_i$ and $H\in\mathcal{L}^n_j$ be such that $i,j\in\{2,\dots,n-1\}$, $i+j\geq n+1$, and $E^{\bot}\subset H$.  Denote $k = i+j-n$, so that $1\le k\le n-2$. Then
$$
\int_{E\cap H} P_{E\cap H}f(w)dw\,\int_{\R^n}f(z)dz\leq{n-k\choose n-i}\int_{E}P_Ef(x)dx\,\int_{H}P_Hf(y)dy.
$$
\end{thm}
If one replaces in Theorem \ref{th:locLWFunctProjSect} (resp.~Theorem \ref{thm: FunctionalLW})
$f$ by $\chi_K$ (resp.~$\exp(-\Vert\cdot\Vert_K)$), $K\in\mathcal K^n$ (resp.~with $0\in K$),  one recovers Theorem \ref{th:localLoomisWhitneyReverse}
(resp.~Theorem \ref{th:localLoomisWhitney}). In order to prove Theorem \ref{thm: FunctionalLW}, we will need to prove a suitable version of Berwald's inequality. This will be done in Section \ref{sec:FunctionalWeigthedBerwald}.


We return once again to the question of finding functional analogues of classical geometric inequalities, but turn our attention to two other inequalities of Rogers and Shephard.
In their paper, Rogers and Shephard linked inequality \eqref{eq:volumesplit} with their classical inequality for the difference body, and with the following generalization which applies to any two convex bodies $K$ and $L$ in $\R^n$   \cite{RS57} (here $A+B = \{ a+b:a\in A, b\in B\}$ is the Minkowski sum of two bodies).
\begin{equation}\label{eq:RogersShephard}
 \vol_n(K\cap(-L)) \vol_n(K+L)\leq{2n\choose n} \vol_n(K)\vol_n(L),
\end{equation}
where equality holds if and only if $K=-L$ is an $n$-dimensional simplex (cf.~\cite{AlJiVi}).
The special case $K = -L$ is the well known Rogers-Shephard inequality for the difference body.

We next formulate the functional version, which was derived in \cite[Theorem 2.1]{AGJV}, for this inequality, and link it to our inequality of Theorem \ref{th:splittRS}. We need to first recall the functional analogue for Minkowski sum and averaging of bodies. This again is defined using the epi-graphs of the logarithms of the functions, for which we take usual averages in $\R^{n+1}$.
More formally, given $f,g\in\F(\R^n)$, let us define a function of one more variable
\[
f\otimes g:\R^n\times[0,1]\rightarrow[0,\infty),\quad f\otimes g(z,t):=\sup_{z=(1-t)x + ty}f(x)^{1-t}g(y)^{t}.
\]
We shall consider the $t$-level of this function $f\otimes g(\cdot, t)$ as the average of $f$ and $g$ with weights $(1-t)$ and $t$. Note that if $f = \exp(-u)$ and $g = \exp (-v)$ then letting \[ {\epi}(w_t) = (1-t)\epi (u) + t \epi(v)\]
we have $f\otimes g (\cdot , t) = \exp (-w_t)$.

Recall that the classical
Pr\'ekopa-Leindler inequality \cite{Pr,Le} implies that for any $t\in[0,1]$
\begin{equation}\label{eq:PrekopaLeindler}
\int_{\R^n}(f\otimes g)(z,t)dz\geq\left(\int_{\R^n}f(x)dx\right)^{1-t}\left(\int_{\R^n}g(y)dy\right)^t.
\end{equation}
This is considered as the functional analogue of the fact that for $t\in[0,1]$ we have
$\vol_n((1-t)K+tL)\geq\vol_n(K)^{1-t}\vol_n(L)^t$, which is the multiplicative version of Brunn-Minkowski inequality:
\[
\vol_n(K+L)^\frac1n\geq\vol_n(K)^\frac1n+\vol_n(L)^\frac1n.
\]
Here equality holds if and only if $K$ and $L$ are homothetic to each other.

If one uses the sum of the functions, rather than an average, one can ask
\[ \epi (w) = \epi(u) + \epi (v), \]
and then $\exp (-w) = f\star g$ is called the Asplund product of $f$ and $g$, and can be written as
\[
f\star g:\R^n\to[0,\infty),\quad f\star g(z):=\sup_{z=x+y}f(x)g(y).
\]

In \cite{AGJV} it was shown that
\begin{equation}\label{eq:FuncRS}
\Vert f*g\Vert_{\infty}\int_{\R^n}f\star g(z)dz\leq{2n\choose n}\Vert f\Vert_{\infty}\Vert g\Vert_{\infty}\int_{\R^n}f(x)dx\int_{\R^n}g(y)dy.
\end{equation}
where as usual the convolution of two functions is $f*g(z):=\int_{\R^n}f(x)g(z-x)dx$. Note that for $f = \chi_K$ and $g = \chi_T$ we have that $f\star g = \chi_{K+T}$ and $\Vert f*g\Vert_{\infty} = \max_z \vol_{n}(K \cap (z-T))$.
Moreover, equality in \eqref{eq:FuncRS} holds if and only if both functions are constant multiples of characteristic functions of an
$n$-dimensional simplex $\Delta$ and of $-\Delta$, respectively.

The union of the averages $(1-t)K + t L$ over $t\in [0,1]$ produces the convex hull of $K$ and $L$. This fact, together with the inequality \eqref{eq:volumesplit}, was used by Rogers and Shephard  in \cite[(16)]{RS58}
 to get a sharp inequality for the convex hull of two bodies with a common point. They showed that
for any $K,L\in\K^n$ with $0\in K\cap L$, one has
\begin{equation}\label{eq:RoShconvhull}
\vol_n(K\cap(-L))\vol_n(\conv(K,L)) \leq 2^n\vol_n(K)  \vol_n(L),
\end{equation}
and equality holds if and only if $K=-L$ is an $n$-dimensional simplex (see \cite{AGJV}). Inequality \eqref{eq:RoShconvhull}
was strengthened in \cite[Theorem 1.6]{AEFO} and \cite[Theorem 2.4]{AGJV} by showing that for any $K,L\in\K^n$ with $0\in K\cap L$, then
\begin{equation}\label{eq:RoShpolar}
\vol_n(\conv(K,L))\vol_n((K^\circ-L^\circ)^\circ)\leq\vol_n(K)\vol_n(L),
\end{equation}
and equality holds if and only if $K$ and $L$ are simplices, with a common vertex at the origin, and such that
the $n$ facets of $K$ and $-L$ containing the origin are contained in the same hyperplanes (cf.~\cite[Theorem 2.4]{AGJV}).

Let us recall yet a third classical inequality of Rogers and Shephard which will come up in our functional constructions: for $K,L\in\K^n$, we have
\begin{equation}\label{eq:RoShHyperplanes}
\vol_n (K\cap(-L))\vol_{n+1} (\conv\{K\times\{0\},L\times\{1\}\})\leq\frac{2^n}{n+1}\vol_n(K)  \vol_n(L),
\end{equation}
and equality holds if and only if $K=-L$ is an $n$-dimensional simplex (cf.~\cite{AEFO}).

Our functional analogue for the body $\conv\{K\times\{0\},L\times\{1\}\}$ is the functional $f\otimes g$.
The convex hull of two functions is then the projection of $f\otimes g$ on the first $n$-coordinates. In terms of epi-graphs, letting $f = \exp (-u)$ and $g = \exp (-v)$ we have that their functional convex hull $\exp(-w)$ satisfies $\epi(w) = \conv(\epi(u), \epi(v))$. Equivalently we can write $\exp(-w)$ as
\[
(f\tilde{\star}g)(z):=\sup_{z=tx+(1-t)y}f(x)^tg(y)^{1-t}=\sup_{t\in[0,1]}f\otimes g(z,t),
\]
which assures $f\tilde{\star}g\in\F(\R^n)$ and $\supp (f\tilde{\star}g)=\conv\{\supp (f),\supp (g)\}$.
In particular, $\chi_K\tilde{\star}\chi_L=\chi_{\conv(\{K, L\})}$.
Moreover, let us observe that for any $f,g,h\in\mathcal{F}(\R^n)$, we have that
\begin{equation}\label{eq:prodishull}
f(z),g(z)\leq h(z)\,\forall z\in\R^n\quad\text{ implies }\quad f\widetilde{\star}g(z)\leq h(z)\,\forall z\in\R^n.
\end{equation}

A functional analogue for inequality  \eqref{eq:RoShconvhull} and \eqref{eq:RoShpolar} in the case $L=K$ was first given by Colesanti in \cite{Col} and afterwards in \cite{AGJV,AEFO} in the general case. In fact, inequality \eqref{eq:RoShpolar} and the characterization of the equality case was obtained as a consequence of this functional inequality. Namely,
\begin{equation}\label{eq:RSFuncAGJV2}
\int_{\R^n}\sqrt{f(w)g(-w)}dw\int_{\R^n}\sqrt{f\star g(2z)}dz\leq
2^n\int_{\R^n}f(x)dx\int_{\R^n}g(y)dy,
\end{equation}
with equality if and only if $\frac{f(x)}{\Vert f\Vert_\infty}=\exp(-\Vert x-x_0\Vert_K)$ and $\frac{g(-x)}{\Vert g\Vert_\infty}=\exp(-\Vert x-x_0\Vert_{L})$, where $x_0\in\R^n$ and $K$ and $L$ are two simplices with the origin as  a common vertex and the $n$ facets containing the origin lying in the same set of $n$ hyperplanes. In Section \ref{sec:Rogers-Shephard type inequalities} we obtain inequalities \eqref{eq:FuncRS} and \eqref{eq:RSFuncAGJV2} as a direct consequence of Theorems \ref{th:splittRS} and  \ref{SplittingColesanti} respectively and show that the characterization of the equality cases in \eqref{eq:RoShpolar} also implies the characterization of the equality cases in \eqref{eq:RSFuncAGJV2}. In addition, we use Theorem \ref{th:splittRS} to obtain
a new extension of \eqref{eq:RoShpolar}.

\begin{thm}\label{thm:RSFuncconvexhullbest}%
	Let $f,g\in\F(\R^n)$ be such that $f(0)=\Vert f\Vert_{\infty}$ and $g(0)=\Vert g\Vert_{\infty}$.
	Then
	\begin{equation}\label{eq:RSFuncconvexhullbest}
	\begin{split}
	&\int_{\R^n}f\tilde{\star}g(z)dz\int_{\R^n}\sup_{0<s<1}\left\{f\left(\frac{x}{1-s}\right)^{1-s}g\left(-\frac{x}{s}\right)^s\right\}dx\leq \\
	&(2n+1){2n\choose n}
\frac{\log\frac{\Vert f\Vert_\infty}{\Vert g\Vert_\infty}}{\Vert f\Vert_\infty-\Vert g\Vert_\infty}\max\{\Vert f\Vert_\infty,\Vert g\Vert_\infty\}^2	
\int_0^1t^n(1-t)^n\left( \int_{\R^n}f(x)^tdx\int_{\R^n}g(y)^{1-t}dy\right)dt.
	\end{split}
	\end{equation}
Equality holds if and only if $\Vert f\Vert_\infty=\Vert g\Vert_\infty$,
$f(x)=\Vert f\Vert_\infty\chi_K(x)$, $g(x)=\Vert g\Vert_\infty\chi_{-L}(x)$, and $K$ and $L$
are simplices with a common vertex at the origin such that the n facets of each of them
containing the origin are contained in the same n hyperplanes.
\end{thm}
Replacing $f$ and $g$ by the characteristic functions of the two convex bodies above, the inequality \eqref{eq:RoShpolar} is recovered.
Let us observe that whenever we use the expression $r(a,b):=\frac{a-b}{\log\frac{a}{b}}$, for any $a,b>0$, we assume that $r(a,a)=a$.

The functional counterpart of inequality \eqref{eq:RoShHyperplanes} is given in the next theorem.

\begin{thm}\label{thm:RSFunc(n+1)}
Let $f,g\in\mathcal{F}(\R^n)$. Then
\begin{equation}\label{eq:RSFunc(n+1)}
\begin{split}
&\int_{\R^{n+1}}f\otimes g(z,t)dzdt\int_{\R^n}\sqrt{f(x)g(-x)}dx\leq\\
&2^n{2n+1\choose n}
\max\{\Vert f\Vert_{\infty},\Vert g\Vert_{\infty}\}\int_0^1t^n(1-t)^n\left(\int_{\R^n}f(x)^tdx\int_{\R^n}g(y)^{1-t}dy\right)dt.
\end{split}
\end{equation}
Equality holds if and only if $\frac{f(x)}{\Vert f \Vert_\infty}=\chi_K(x)=\frac{g(-x)}{\Vert g\Vert_\infty}$ with $K$ being an $n$-dimensional simplex.
\end{thm}
If we substitute $f$ and $g$ by the characteristic functions of two convex bodies we recover \eqref{eq:RoShHyperplanes}.
%
%
%

The paper is organized as follows: Section \ref{sec:Estimates by marginals} is divided in two subsections. In the first one we prove Theorems \ref{th:splittRS} and \ref{SplittingColesanti}, which are both extensions of \eqref{eq:volumesplit} and give lower estimates for the integral of a log-concave function in terms of the integral of projections and sections by orthogonal subspaces. In the second one we prove Theorems \ref{th:upperboundcharact} and \ref{th:upperboundexponent}, which give upper bounds for the integral of a log-concave function in terms of the integrals of its sections or projections.
For convenience of the reader, we have moved all the proofs for equality cases in the various theorems into one section, towards the end of the paper, Section \ref{sec:Equality cases}.
In Section \ref{sec:FunctionalWeigthedBerwald} we will prove a Berwald's inequality that is needed in the proof of Theorem \ref{thm: FunctionalLW}.
In Section \ref{sec:LW} we describe the setting of Loomis-Whitney type results. We divide this section into two subsections. The first one is devoted to prove Theorems \ref{th:localLoomisWhitneyReverse} and \ref{th:locLWFunctProjSect} which are geometric and functional reverse local Loomis-Whitney type results which will be deduced from the results in Section \ref{sec:Estimates by marginals}. The second one is devoted to prove Theorems \ref{th:localLoomisWhitney} and \ref{thm: FunctionalLW}, which are direct geometric and functional local Loomis-Whitney results.
In Section \ref{sec:Rogers-Shephard type inequalities} we recover and prove some new functional Rogers-Shephard type inequalities. Section \ref{sec:Equality cases} is devoted to the study of the equality cases in the inequalities previously proven.   
Finally, Section \ref{sec:appendixBerwald} is an appendix in which we give the proof of the classical Berwald's inequality, together with the equality cases, that is used in the proof of Theorem \ref{th:localLoomisWhitney}. We include this appendix since we were not able to find a translation of this result into English in the literature.

\section{Estimates for the integral of a log-concave function by its marginals}\label{sec:Estimates by marginals}

Let us denote the Euclidean norm of $x\in\R^n$ by $\|x\|_2=\sqrt{x_1^2+\cdots+x_n^2}$
and the unit Euclidean ball by $B^n_2:=\{x\in\R^n:\|x\|_2\leq 1\}$.
For $f\in\mathcal{F}(\R^n)$ and $H\in\mathcal{L}^n_i$ we define the symmetric
function of $f$ with respect to $H$, extending the Steiner and Schwarz symmetrals studied in \cite{CoSaYe}.
We recall that for $H\in\mathcal{L}^n_i$ and a set $K$ with measurable section $K\cap(x+H^\bot)$, the symmetral of $K$ with respect to $H$ is the set defined by
\[
S_H(K):=\bigcup_{x\in H}\left\{x+\rho_x B_2^n\cap H^{\bot}:\vol(\rho_x B_2^n\cap H^{\bot})=\vol(K\cap(x+H^{\bot}))\right\}.
\]
The fact that when $K$ is convex, so is $S_HK$, follows from Brunn's concavity principle.
Notice that if $K$ is a convex body and $\dim(H)=n-1$ (resp.~$\dim(H)=1$) then $S_H(K)$ is the Steiner (resp.~Schwarz) symmetrization of $K$ with respect to $H$.

\begin{df}
 Let $f=e^{-u}\in\mathcal{F}(\R^n)$ and $H\in\mathcal{L}^n_i$.
	Denote $epi(u)=\{(x,t)\in\R^{n}\times [0,\infty):u(x)\geq t\}$ and $\overline{H}:={\textrm{span}}\{H, e_{n+1}\}\in\mathcal{L}^{n+1}_{i+1}$. Then $S_H(f)\in \mathcal{F}(\R^n)$ is defined via its epi-graph by
\[
\epi(-\log S_H(f))=S_{\overline{H}}(\epi(-\log f)).
\]
The fact that $S_Hf$ is log-concave follows from the fact that $S_{\overline{H}}(\epi(-\log f))$ is an epi-graph and that
$S_{\overline{H}}$ preserves convexity, by Brunn's concavity principle.
\end{df}

\begin{rmk}
Note that by definition, for $f=\exp(-u)\in\mathcal{F}(\R^n)$,   $H\in\mathcal{L}^n_i$,  we have
  for any $(x,y)\in H\times H^{\bot}$   that $S_H(f)(x+\bar{y})=S_H(f)(x+y)$ for every $\bar{y}$ with $\Vert\bar{y}\Vert_2=\Vert y\Vert_2$.
Furthermore, notice that for every $x\in H$ we have
\begin{equation}\label{eq:projequalssymm}
P_Hf(x)=S_Hf(x).
\end{equation}
Indeed,
\begin{equation*}
\begin{split}
S_Hf(x) & =\exp(-\inf\{t\in\R:(x,t)\in S_{\overline{H}}(epi(u))\})\\
& =\exp(-\inf\{t\in\R:(x+y,t)\in epi(u),\,\textrm{for some }y\in H^{\bot}\})\\
& =\exp(-\inf_{y\in H^{\bot}}\{u(x+y)\})=P_Hf(x).
\end{split}
\end{equation*}

\noindent We shall also use the fact that $S_H$ preserves integrals on the fibers $x_1+H^\bot$, namely
\begin{equation}\label{eq:integsymmsubspace}
\begin{split}
\int_{x_1+H^{\bot}}f(y)dy=
\int_{x_1+H^{\bot}}e^{-u(y)}dy=
\int_0^\infty\vol_{n-i}\left(\{y\in x_1+H^{\bot}:e^{-u(y)}\geq t\}\right)dt=&\\
\int_0^\infty\vol_{n-i}\left(\{y\in x_1+H^{\bot}:e^{-v(y)}\geq t\}\right)dt=
\int_{x_1+H^{\bot}}S_Hf(y)dy,&
\end{split}
\end{equation}
where $e^{-v}=S_Hf$, and
\begin{equation}\label{eq:integsymm}
\int_{\R^n}f(z)dz=\int_H\int_{x+H^{\bot}}f(y)dydx = \int_H\int_{x+H^{\bot}}S_Hf(y)dydx = \int_{\R^n} S_Hf(z)dz.
\end{equation}
\end{rmk}

\subsection{Lower bounds for the integral of $f$}

This subsection is devoted to prove Theorems \ref{th:splittRS} and \ref{SplittingColesanti}, which are both extensions of \eqref{eq:volumesplit} and give lower estimates for the integral of a log-concave function in terms of the integrals of projections and sections by orthogonal subspaces.

The following lemma is an easy consequence of a very similar lemma of Rogers and Shephard \cite[Lemma]{RS58}, which is the main ingredient in the proof of \eqref{eq:volumesplit}. Their case was the equality case in the Lemma below, namely when $y\in K$ and $x\in L$. While a simple proof using \eqref{eq:volumesplit} can be easily given, we chose to give a proof which is along the original line of proof for \eqref{eq:volumesplit} for the case in which the point $(x,y)$ considered in the statement of the theorem is possibly not contained in $L\times K$.

	\begin{lemma}\label{Ortholemma}
		Let $K\in\K^i$, $L\in\K^m$, $i,m\in\N$. Then for any $x\in \R^m$ and $y\in \R^i$ we have
		\[
		\vol_{i+m}\left( \conv\{K\times\{x\} ,\{y\} \times L\}\right)\geq \binom{i+m}{i}^{-1}\vol_{i } ( K) \vol_m(L).
		\]
		Equality holds if and only if $x\in L$ and $y\in K$, or either $K$ or $L$ has empty interior, relative to $\R^i$ or $\R^m$ respectively.
	\end{lemma}

\begin{proof}
 Note that if $x\in L$ and $y\in K$ then we have equality by  \eqref{eq:volumesplit}.
		We use Shephard's result on the convexity of the volume of a system of moving shadows \cite{Shep}, as follows: Fix some $x_1$ in the interior of $L$ and $y_1$ in the interior of $K$, and define the vector $v= (y_1-y, x-x_1)\in \R^{i+m}$. Consider the
		function \[ f(t) = \vol_{i+m}\left( \conv\{K\times\{x\}  - tv,\{y\} \times L+tv\}\right) .\]
		By Shephard's result \cite{Shep}  we know that $f$ is convex in $t$. For $t= 1/2$ we have that
		\[ f(1/2) = \vol_{i+m}\left( \conv\left\{\left(K-\frac{y_1-y}{2}\right)\times\left\{\frac{x+x_1}{2}\right\},\left\{\frac{y_1+y}{2}\right\} \times \left(L+\frac{x-x_1}{2}\right)\right\}\right),\]
		so that in particular, since
 \[
 \left(\frac{y_1+y}{2},\frac{x+x_1}{2}\right)\in
 \left(\left(K-\frac{y_1-y}{2}\right)\times\left\{\frac{x+x_1}{2}\right\}\right)\cap\left(\left\{\frac{y_1+y}{2}\right\} \times \left(L+\frac{x-x_1}{2}\right)\right),
 \]
 we have that $f(1/2) = \binom{i+m}{i}^{-1}\vol_{i}(K) \vol_m(L)$. Moreover, this equality holds true also for $f(t)$ for $t$ at an interval around $1/2$, since the two bodies continue to intersect.
		However, a convex function can be constant only on the set where it attains its minimum, which concludes the proof of the inequality.
The proof for the equality case is given in Subsection \ref{subsect:equality in lemma Ortholemma} below. 	
\end{proof}

\begin{lemma}\label{lem:MinLowerBound}
Let  $f\in\F(\R^i)$, $g\in\F(\R^m)$,  and let $A$ and $B$  be real numbers such that $A\geq\Vert f\Vert_\infty$ and $B\geq\Vert g\Vert_\infty$. Then
\[
\int_{\R^{n+m}}\min\left\{\frac{f(x)}{A},\frac{g(y)}{B}\right\}dxdy\geq
\int_{\R^n}\frac{f(x)}{A}dx\int_{\R^m}\frac{g(y)}{B}dy.
\]
Equality holds if and only if and only if $\frac{f(x)}{A}=\frac{f(x)}{\Vert f\Vert_\infty}=\chi_K(x)$ or $\frac{g(y)}{B}=\frac{g(y)}{\Vert f\Vert_\infty}=\chi_L(y)$ for some convex body $K\subseteq\R^i$ or $L\subseteq\R^m$.
\end{lemma}

\begin{proof}
The inequality is trivial since, as both $\frac{f(x)}{A}$ and $\frac{g(y)}{B}$ belong to the interval $[0,1]$, their minimum is greater than or equal to their product. It is clear that if $\frac{f(x)}{A}$ or $\frac{g(y)}{B}$ is a characteristic function, then there is equality. Assume now that there is equality in this inequality and assume that $\frac{g(y)}{B}$ is not a characteristic function. Then, since log-concave functions are continuous in the interior of their supports, there exists $y_0\in \R^m$ and a neighborhood $U$ of $y_0$ such that $0<\frac{g(y)}{B}<1$ for every $y\in U$. Consequently, for almost every $x\in\R^i$, we have that $\frac{f(x)}{A}=0,1$ and then $\frac{f(x)}{A}$ is a characteristic function.
\end{proof}

The next lemma is an extension of Lemma \ref{Ortholemma} to log-concave functions. We are given two log-concave functions, one on $\R^i$ and another on $\R^m$, which are then considered as orthogonal subspaces of $\R^{i+m}$. We build from them the ``convex hull'' function and estimate its integral from below by the integrals of $f$ and $g$.

\begin{lemma}\label{orthoRS}
Let $f\in\mathcal{F}(\R^i)$ and $g\in\mathcal{F}(\R^m)$, $i,m\in\N$. Let us define for any $z_1\in\R^i$ and $z_2\in\R^m$ $F(z_1,z_2):=f(z_1)\chi_{\{0\}}(z_2)$
and $G(z_1,z_2):=g(z_2)\chi_{\{0\}}(z_1)$.
Then
\begin{equation}\label{eq:ProjSect}
\max\{\Vert f\Vert_\infty, \Vert g\Vert_\infty\}\int_{\R^{i+m}}F\tilde{\star}G(z)dz\geq{m+i\choose i}^{-1}
\int_{\R^i}f(x)dx\int_{\R^m}g(y)dy.
\end{equation}
Equality holds if and only if $\frac{f}{\Vert f\Vert_\infty}=\chi_K$ and $\frac{g}{\Vert g\Vert_\infty}=\chi_L$ for some convex bodies $K\in\mathcal{K}^i$, $L\in\mathcal{K}^m$, with $0\in K$ and $0\in L$, and such that $\Vert f\Vert_{\infty}=\Vert g\Vert_{\infty}$.
\end{lemma}

\begin{proof}
Notice that  by definition of the operation $\tilde{\star}$ and of $F$ and $G$, we have that
\[ F\tilde{\star} G (z_1, z_2) =  \sup_{\theta \in (0,1)} f\left(\frac{z_1}{1-\theta}\right)^{1-\theta} g\left(\frac{z_2}{\theta}\right)^\theta \ge \sup_{\theta \in (0,1)} \min \left\{ f\left(\frac{z_1}{1-\theta}\right),  g\left(\frac{z_2}{\theta}\right)\right\}. \]
In particular, setting $A:=\Vert F\tilde{\star} G\Vert_\infty = \max\{\Vert f\Vert_\infty, \Vert g\Vert_\infty\}$,
we compute
\begin{equation*}
\begin{split}
&\frac1{A}\int_{\R^{i+m}} F\tilde{\star}G(z)dz \\
&=\int_0^1 \vol_n\left(\left\{(z_1,z_2):\sup_{0<\theta<1}f\left(\frac{z_1}{\theta}\right)^{\theta}
g\left(\frac{z_1}{1-\theta}\right)^{1-\theta}\geq tA \right\}\right)dt\\
& \stackrel{(a)}{\geq}\int_0^1 \vol_n\left(\left\{(z_1,z_2):\sup_{0<\theta<1}\min \left\{f\left(\frac{z_1}{1-\theta}\right),  g\left(\frac{z_2}{\theta}\right)\right\}\geq tA\right\}\right)dt\\
& =\int_0^1\vol_n\left( \conv\left(\left\{(x,0)\in\R^n:f(x)\geq tA\right\}\cup\left\{(0,y)\in\R^n:g(y)\geq tA\right\}\right)\right)dt \\
& \stackrel{(b)}{\geq}{i+m\choose i}^{-1}\int_0^1\vol_i\left(\left\{x\in\R^i:f(x)\geq tA\right\}\right)\vol_m\left(\left\{y\in\R^m:g(y)\geq tA\right\}\right)dt \\
& ={i+m\choose i}^{-1}\int_0^1\left[\int_{\R^i}\chi_{\{f(x)\geq tA\}}(x)dx\int_{\R^m}\chi_{\{g(y)\geq tA\}}(y)dy\right]dt \\
& ={i+m\choose i}^{-1}\int_{\R^i}\int_{\R^m}\min\left\{\frac{f(x)}{A},\frac{g(y)}{A}\right\}dydx\\
& \stackrel{(c)}{\geq}
{i+m\choose i}^{-1}\int_{\R^i}\frac{f(x)}{A}dx\int_{\R^m}\frac{g(y)}{A}dy,
\end{split}
\end{equation*}
where (a) follows from the trivial inclusion between the sets, (b) follows from Lemma \ref{Ortholemma} and (c) follows from Lemma \ref{lem:MinLowerBound}.

The proof of the equality case will appear in Subsection \ref{subsub:orthoRS}.
\end{proof}

Before using Lemma \ref{orthoRS} to prove Theorem \ref{th:splittRS}, let us show an analogous result to Lemma \ref{orthoRS}, when considering the function $F\otimes G(\cdot,t)$ for a fixed $t$ instead of $F\tilde{\star}G$, which is obtained taking the supremum in $t$.
\begin{lemma}\label{lem:linPL}
Let $f\in\F(\R^i)$, $g\in\F(\R^m)$.
Let us define for any $z_1\in\R^i$ and $z_2\in\R^m$,  $F(z_1,z_2):=f(z_1)\chi_{\{0\}}(z_2)$ and $G(z_1,z_2):=g(z_2)\chi_{\{0\}}(z_1)$,
then for every $t\in[0,1]$
\[
\Vert f\Vert_\infty^{t}\Vert g\Vert_\infty^{1-t}\int_{\R^{i+m}}F\otimes G(z,t)dz \ge  t^{m}(1-t)^i\int_{\R^i}f (x)dx\int_{\R^m}g (y)dy.
\]
\end{lemma}

\begin{proof}
Notice that for every $t\in[0,1]$ and every $z_1\in\R^i$, $z_2\in\R^m$
\[ F\otimes G (z_1, z_2,t) =   f\left(\frac{z_1}{1-t}\right)^{1-t} g\left(\frac{z_2}{t}\right)^t. \]
In particular, $
\Vert F\otimes G(\cdot,t)\Vert_\infty =\Vert f\Vert_\infty^{1-t}\Vert g\Vert_\infty^{t}.
$

Therefore
\begin{eqnarray*}
	 \int_{\R^{i+m}} F\otimes G(z,t)dz& = &
	\int_{\R^{i}}
	 f\left(\frac{z_1}{1-t}\right)^{1-t}dz_1 \int_{\R^m}g\left(\frac{z_2}{t}\right)^t dz_2\\
	 & = & t^{m}(1-t)^i\int_{\R^i}f(x)^{1-t}dx\int_{\R^m}g(y)^t dy\\
	 & \ge & \Vert f\Vert_\infty^{1-t}\Vert g\Vert_\infty^{t} t^{m}(1-t)^i\int_{\R^i}\frac{f(x)}{\Vert f\Vert_\infty}dx\int_{\R^m}\frac{g(y)}{\Vert g\Vert_\infty}dy.
 \end{eqnarray*}

\end{proof}

\begin{rmk}\label{rmk:worsebound}
Notice that from the latter inequality one can easily deduce that
\[
\max\{\Vert f\Vert_\infty,\Vert g\Vert_\infty\}\int_{\R^{i+m}}F\tilde{\star}G(z)dz \geq
\frac{i^im^m}{(i+m)^{i+m}}\int_{\R^i}f(x)dx\int_{\R^m}g(y)dy,
\]
which is slightly worse than inequality \eqref{eq:ProjSect} in Lemma \ref{orthoRS}.
Indeed
\begin{equation*}
\begin{split}
 \max\{\Vert f\Vert_\infty,\Vert g\Vert_\infty\}\int_{\R^{i+m}}F\tilde{\star}G(z)dz& =\max\{\Vert f\Vert_\infty,\Vert g\Vert_\infty\}\int_{\R^{i+m}}\sup_{t\in[0,1]}F\otimes G(z,t)dz\\
& \geq \sup_{t\in[0,1]} \left( \Vert f\Vert_\infty^{t}\Vert g\Vert_\infty^{1-t}  \int_{\R^{i+m}}F\otimes G(z,t)dz\right) \\
& \geq \sup_{t\in[0,1]}  t^i(1-t)^{m} \int_{\R^i}f(x)dx\int_{\R^m}g(y)dy\\
& = \frac{i^im^m}{(i+m)^{i+m}}\int_{\R^i}f(x)dx\int_{\R^m}g(y)dy.
\end{split}
\end{equation*}
\end{rmk}

We are now in a position to prove our main Theorem \ref{th:splittRS}, which also serves as the main tool for the proof of the reverse local Loomis-Whitney inequality as well as the proof of other Rogers-Shephard type inequalities.

\begin{proof}[Proof of Theorem \ref{th:splittRS}]
Let us assume, without loss of generality 
 that $H=\textrm{span}\{e_1,\dots, e_i\}$ and $H^\perp=\textrm{span}\{e_{i+1},\dots,e_n\}$ (otherwise consider $f\circ U$ for a suitable $U\in O(n)$). We shall
consider the symmetral $S_H f$ restricted to $H$ and to $H^\perp$ respectively, and use these two as the functions to which
Lemma \ref{orthoRS} is applied. More precisely, for $(x,y)\in H\times H^{\bot}=\R^i\times \R^{n-i}$ define
$F(x,y):=S_Hf(x)\chi_{\{0\}}(y)$ and $G(x,y):=S_Hf(y)\chi_{\{0\}}(x)$ (where we have identified $x = (x,0)$ and $y = (0,y)$).
Then clearly $\max\{\Vert S_Hf|_H\Vert_\infty, \Vert S_Hf|_{H^\bot}\Vert_\infty\} = \|S_Hf\|_\infty = \|f\|_\infty$.
By Lemma \ref{orthoRS} we thus know that
\[ \|f\|_\infty \int_{\R^{n}}F\tilde{\star}G(z)dz\geq{n\choose i}^{-1}
\int_{H}S_Hf(x)dx\int_{H^\bot}S_Hf(y)dy. \]
Since both $F$ and $G$ are bounded from above by $S_Hf$ we have by \eqref{eq:prodishull} that
$F\tilde{\star}G \le S_H f$ on all of $\R^n$. Thus, using \eqref{eq:integsymm} we see that
\begin{eqnarray*}
\Vert f\Vert_{\infty} \int_{\R^n}f(z)dz & = & \Vert f\Vert_{\infty} \int_{\R^n}S_Hf(z)dz\\
&{\geq}& \Vert f\Vert_{\infty}  \int_{\R^n}F\tilde{\star}G(z)dz\\
&{\geq}& {n\choose i}^{-1} \int_{H}S_Hf(x)dx \int_{H^\bot }S_Hf(y)dy.
\end{eqnarray*}
Finally, using \eqref{eq:integsymmsubspace} we can see that $\int_{H^\bot }S_Hf(y)dy = \int_{H^\bot}f(y)dy$ and
\eqref{eq:projequalssymm} to notice that the restriction to $H$ of $S_Hf$ equals $P_Hf$, concluding
\[
	\Vert f\Vert_{\infty} \int_{\R^n}f(z)dz
\ge {n\choose i}^{-1} \int_{H}P_Hf(x)dx \int_{H^\bot }f(y)dy,
\]
as claimed. The equality case will be proven in Subsection \ref{subsub:splittRS}.
\end{proof}

The next lemma is the first step for the proof of Theorem \ref{SplittingColesanti} before symmetrizing $f$ in order to obtain the projection of $f$ onto $H$ instead of its section.


\begin{lemma}\label{th:SplittColesLambda}
Let $f\in\mathcal{F}(\R^n)$ be such that $f(0)=\Vert f\Vert_{\infty}$, $\lambda\in(0,1)$, and $H\in\mathcal{L}^n_i$. Then
\begin{equation}\label{splittingColLambda}
\int_{\R^n}f(z)dz \ge (1-\lambda)^i\lambda^{n-i} \int_{H} f(x)^{1-\lambda}dx \int_{H^\perp} f(y)^{\lambda}dy.
\end{equation}
Moreover, equality holds if and only if for every $z=(x,y)\in H\times H^\perp$, $\frac{f(z)}{\Vert f\Vert_\infty}=\,\exp(-\Vert x\Vert_{K}-\Vert y\Vert_{L})$ for some $K\subseteq H$ and $L\subset H^\bot$ such that $0\in K\times L$.
\end{lemma}

\begin{proof}
For any $z=(z_1,z_2)\in H\times H^\bot$, the log-concavity of $f$ implies that
\begin{equation*}
\begin{split}
f(z) & =f\left((1-\lambda)\frac{z_1}{1-\lambda},\lambda\frac{z_2}{\lambda}\right)=f\left((1-\lambda)\left(\frac{z_1}{1-\lambda},0\right)+\lambda\left(0,\frac{z_2}{\lambda}\right)\right)\\
& \geq f\left(\frac{z_1}{1-\lambda},0\right)^{1-\lambda}f\left(0,\frac{z_2}{\lambda}\right)^\lambda,
\end{split}
\end{equation*}
and thus
\[
\begin{split}
\int_{\R^n}f(z)dz & \geq\int_{\R^n}f\left(\frac{z_1}{1-\lambda},0\right)^{1-\lambda}f\left(0,\frac{z_2}{\lambda}\right)^\lambda dz_1dz_2 \\
& =\int_H f\left(\frac{z_1}{1-\lambda},0\right)^{1-\lambda}dz_1 \int_{H^\bot} f\left(0,\frac{z_2}{\lambda}\right)^\lambda dz_2\\
& =(1-\lambda)^i\lambda^{n-i}\int_Hf(x)^{1-\lambda}dx\int_{H^\bot}f(y)^\lambda dy.
\end{split}
\]

The proof of the equality case is done in Subsection \ref{subsub:lemasplittcoleslambda}.
\end{proof}

\begin{proof}[Proof of Theorem \ref{SplittingColesanti}]
Let us define $\tilde{f}:=S_Hf$. By Lemma \ref{th:SplittColesLambda} we have that
\[
\begin{split}
\int_{\R^n}f(z)dz & =\int_{\R^n}\tilde{f}(z)dz\geq(1-\lambda)^i\lambda^{n-i} \int_H\tilde{f}(x)^{1-\lambda}dx\int_{H^\bot}\tilde{f}(y)^\lambda dy\\
& =(1-\lambda)^i\lambda^{n-i} \int_HP_Hf(x)^{1-\lambda}dx\int_{H^\bot}f(y)^\lambda dy,
\end{split}
\]
as desired.

The proof of the equality case can be found in Subsection \ref{subsub:thsplittcoleslambda}.
\end{proof}

\begin{rmk}
In Theorem \ref{SplittingColesanti} with the function $f(z)=\exp(-\Vert z\Vert_K)$, for some $K\in\K^n$ with $0\in K$, we
immediately obtain that
\[
\begin{split}
n!\vol_n(K) & =\int_{\R^n}f(z)dz \geq (1-\lambda)^i\lambda^{n-i}\int_H P_Hf(x)^{1-\lambda}dx\int_{H^\bot}f(y)^\lambda dy\\
&= (1-\lambda)^i\lambda^{n-i} \int_H e^{-\Vert x\Vert_{\frac{P_HK}{1-\lambda}}}dx \int_{H^\bot} e^{-\Vert y\Vert_{\frac{K\cap H^{\bot}}{\lambda}}}dy\\
&=(1-\lambda)^i\lambda^{n-i} i!\vol_i\left(\frac{P_HK}{1-\lambda}\right)(n-i)!\vol_{n-i}\left(\frac{K\cap H^\bot}{\lambda}\right)\\
&=i!(n-i)!\vol_i(P_HK)\vol_{n-i}(K\cap H^\bot),
\end{split}
\]
hence implying \eqref{eq:volumesplit}.
\end{rmk}

\subsection{Upper bounds for the integral of $f$}

This section is devoted to prove Theorems \ref{th:upperboundcharact} and \ref{th:upperboundexponent}, which give upper bounds for the integral of a log-concave function in terms of the integrals of its sections or projections.
We start by showing the proof of Theorem \ref{th:upperboundcharact}, which is, like in the case of \eqref{eq:upperbound}, a direct application of Fubini's theorem.

\begin{proof}[Proof of Theorem \ref{th:upperboundcharact}]
\begin{equation*}
	\begin{split}
	\int_{\R^n}f(z)dz&=\int_{H}\int_{x_0+H^{\bot}}f(y)dydx_0\\
	&\leq\int_H\int_{x_0+H^{\bot}}f(y)^{1-\lambda}\max_{x\in x_0+H^{\bot}}f(x)^{\lambda}dydx_0\\
	&\leq\int_H\max_{x\in x_0+H^{\bot}} f(x)^{1-\lambda}dx_0\max_{x_0\in H}\int_{x_0+H^{\bot}}f(y)^{\lambda}dy\\
	&=\int_HP_Hf^{1-\lambda}(x_0)dx_0\max_{x_0\in H}\int_{x_0+H^{\bot}}f|_{x_0+H^{\bot}}^{\lambda}(y)dy.
	\end{split}
	\end{equation*}
\end{proof}

The next lemma consists of a reverse inequality to Lemma \ref{lem:MinLowerBound}.
\begin{lemma}\label{lem:MinUpperBound}
Let $f\in\mathcal F(\R^n)$ and $g\in\mathcal F(\R^m)$.
Then
\[
\int_{\R^{n+m}}\min\left\{\frac{f(x)}{\Vert f\Vert_\infty},\frac{g(y)}{\Vert g\Vert_\infty}\right\}dxdy\leq
{n+m\choose n}\int_{\R^n}\frac{f(x)}{\Vert f\Vert_\infty}dx\int_{\R^m}\frac{g(y)}{\Vert g\Vert_\infty}dy.
\]
Equality holds if and only if $\frac{f(x)}{\Vert f\Vert_\infty}=\exp(-\Vert x-x_0\Vert_K)$ and $\frac{g(y)}{\Vert g\Vert_\infty}=\exp(-\Vert y-y_0\Vert_L)$ for $x_0\in\R^n$, $y_0\in\R^m$, and some convex bodies $K\subseteq\R^n$ and $L\subseteq\R^m$ with $0\in K$ and $0\in L$.
\end{lemma}

\begin{proof}
Denoting by
\[
K_t:=\{x\in \R^n:f(x)\geq t\Vert f\Vert_\infty\}\quad\text{and}\quad L_t:=\{y\in \R^m:g(y)\geq t\Vert g\Vert_\infty\}
\]
then
\[
\begin{split}
\int_{\R^{n+m}}\min\left\{\frac{f(x)}{\Vert f\Vert_\infty},\frac{g(y)}{\Vert g\Vert_\infty}\right\}dxdy &
=\int_0^1\int_{\R^n}\int_{\R^m}\chi_{K_t}(x)\chi_{L_t}(y)dydxdt \\
& =\int_0^1\vol_n(K_t)\vol_{m}(L_t)dt.
\end{split}
\]
On the other hand, denoting by
\[
C_t:=\{(x,y)\in \R^n\times \R^m:f(x)g(y)\geq t\Vert f\Vert_\infty\Vert g\Vert_\infty\},
\]
and $H=\lin\{e_1,\dots,e_n\}$ we observe that
\[
\begin{split}
P_HC_t & =\{(x,0)\in \R^{n+m}:\max_{y\in \R^m}f(x)g(y)\geq t\Vert f\Vert_\infty\Vert g\Vert_\infty\} \\
& =\{(x,0)\in \R^{n+m}:f(x)\geq t\Vert f\Vert_\infty\}=K_t\times\{0\}^{m}
\end{split}
\]
and that for any $\bar{x}\in K_t$
\[
\begin{split}
C_t\cap(\bar{x}+H^\bot) & =\{(\bar{x},y)\in \R^{n+m}: f(\bar{x})g(y)\geq t\Vert f\Vert_\infty\Vert g\Vert_\infty\}\\
& =\{\bar{x}\}\times L_{t\frac{\Vert f\Vert_\infty}{f(\bar{x})}}.
\end{split}
\]
By Rogers-Shephard inequality \eqref{eq:volumesplit} then
\[
\begin{split}
\vol_{n+m}(C_t) & \geq{n+m\choose n}^{-1}\vol_n(P_HC_t)\max_{\bar{x}\in \R^n}\vol_{m}(C_t\cap(\bar{x}+H^\bot)) \\
& ={n+m\choose n}^{-1}\vol_n(K_t\times\{0\}^{m})\max_{\bar{x}\in \R^n}\vol_{m}(\{\bar{x}\}\times L_{t\frac{\Vert f\Vert_\infty}{f(\bar{x})}})\\
& ={n+m\choose n}^{-1}\vol_n(K_t)\vol_{m}(L_t).
\end{split}
\]
Hence we conclude that
\[
\begin{split}
\int_{\R^{n+m}}\min\left\{\frac{f(x)}{\Vert f\Vert_\infty},\frac{g(y)}{\Vert g\Vert_\infty}\right\}dxdy &
= \int_0^1\vol_n(K_t)\vol_{m}(L_t)dt \\
&\leq {n+m\choose n}\int_0^1\vol_{n+m}(C_t)dt\\
&={n+m\choose n}\int_{\R^{n+m}}\frac{f(x)g(y)}{\Vert f\Vert_\infty\Vert g\Vert_\infty}dxdy\\
&={n+m\choose n}\int_{\R^n}\frac{f(x)}{\Vert f\Vert_\infty}dx \int_{\R^m}\frac{g(y)}{\Vert g\Vert_\infty}dy.
\end{split}
\]

The characterization of the equality case is shown in Subsection \ref{subsub:lemaminupperbound}.
\end{proof}


Now we can prove Theorem \ref{th:upperboundexponent}.

\begin{proof}[Proof of Theorem \ref{th:upperboundexponent}]
For any $z=(x,y)\in H\times H^\bot$ we have that
\[
\frac{f(x,y)}{\Vert f\Vert_\infty}\leq \min\left\{\frac{P_Hf(x)}{\Vert f\Vert_\infty}, \frac{P_{H^\bot}f(y)}{\Vert f\Vert_\infty}\right\},
\]
and by Lemma \ref{lem:MinUpperBound} we obtain
\[
\begin{split}
\int_{\R^n}\frac{f(z)}{\Vert f\Vert_\infty}dz & \leq
\int_{H\times H^\bot}\min\left\{\frac{P_Hf(x)}{\Vert f\Vert_\infty}, \frac{P_{H^\bot}f(y)}{\Vert f\Vert_\infty}\right\}dxdy\\
& \leq {n\choose i}\int_H\frac{P_Hf(x)}{\Vert f\Vert_\infty}dx\int_{H^\bot}\frac{P_{H^\bot}f(y)}{\Vert f\Vert_\infty}dy.
\end{split}
\]

The equality case is shown in Subsection \ref{subsub:thmupperboundexponent}.
\end{proof}

\section{A functional weighted Berwald's inequality}\label{sec:FunctionalWeigthedBerwald}

This section is devoted to prove the following Theorem, which is a version of Berwald's inequality and will be essential for the proof of Theorem \ref{thm: FunctionalLW}.

\begin{thm}\label{thm:BerwaldExponential}
Let $f\in\F(\R^n)$ and let $L$ be the convex set $L=\{(x,t)\in\R^n\times [0,\infty):f(x)\geq e^{-t}\Vert f\Vert_\infty\}$. Let $h_1,\dots,h_m:L\to[0,\infty)$ be continuous, concave, not identically null functions, $\alpha_1,\dots,\alpha_m>0$ and $\sigma=\alpha_1+\dots+\alpha_m$. Then,
$$
\frac{1}{\int_L e^{-t}dxdt}\int_L \prod_{i=1}^mh_i^{\alpha_i}(x,t)e^{-t}dxdt\leq \frac{\Gamma\left(1+\sigma\right)}{\prod_{i=1}^m\Gamma\left(1+\alpha_i\right)}\prod_{i=1}^m \frac{1}{\int_L e^{-t}dxdt}\int_L h_i^{\alpha_i}(x,t)e^{-t}dxdt.
$$
\end{thm}

We will prove a series of lemmas ending up in the proof of the theorem. The proof follows the lines of the version of Berwald's inequality included in the appendix. 

\begin{lemma}\label{logconcave}
Let $h:\R^n\to[0,\infty)$ be a concave function, $g\in\F(\R^n)$. Then
$$
I_h(s)=\int_{K_s(h)}g(x)\,dx
$$
is a log-concave function on $[0,\infty)$, where $K_s(h)=\{x\in\R^n:h(x)\ge s\}$.
\end{lemma}

\begin{proof}
Let $\theta\in[0,1]$, $s_0,s_1\in[0,\infty)$, and $s_\theta=(1-\theta)s_0+\theta s_1$.

The concavity of $h$ gives $K_{s_\theta}(h)\supseteq(1-\theta)K_{s_0}(h)+\theta K_{s_1}(h)$. Then
$$
I_h(s_\theta)
=
\int_{K_{s_\theta}(h)}g(x)\,dx
\ge
\int_{(1-\theta)K_{s_0}(h)+\theta K_{s_1}(h)}g(x)\,dx.
$$
Now, using the log-concavity of $g$, for any $x_0,x_1\in\R^n$,
$$
g(x_\theta)\chi_{(1-\theta)K_{s_0}(h)+\theta K_{s_1}(h)}(x_\theta)
\ge
\left(g(x_0)\chi_{K_{s_0}(h)}(x_0)\right)^{1-\theta}
\left(g(x_1)\chi_{K_{s_1}(h)}(x_1)\right)^{\theta}
$$
where $x_\theta=(1-\theta)x_0+\theta x_1$. Pr\'ekopa-Leindler inequality now gives
$$
\int_{(1-\theta)K_{s_0}(h)+\theta K_{s_1}(h)}g(x)\,dx
\ge
\left(\int_{K_{s_0}(h)}g(x)\,dx\right)^{1-\theta}
\left(\int_{K_{s_1}(h)}g(x)\,dx\right)^{\theta}.
$$
Then $I_h(s_\theta)\ge I_h(s_0)^{1-\theta}I_h(s_1)^\theta$ as desired.
\end{proof}

\begin{lemma}\label{lem:BerwaldExponentialDensity}
Let $f\in\F(\R^n)$ and let $L$ be the convex set $L=\{(x,t)\in\R^n\times [0,\infty):f(x)\geq e^{-t}\Vert f\Vert_\infty\}$. Let $h:L\to[0,\infty)$ be a continuous, concave, not identically null function. Then,
$$
\Phi_\gamma(h):=\left(\frac{1}{\Gamma\left(1+\gamma\right)\int_L e^{-t}dxdt}\int_L h^\gamma(x,t)e^{-t}dxdt\right)^\frac{1}{\gamma}.
$$
is decreasing in $\gamma\in(0,\infty)$.
\end{lemma}
\begin{proof}
For any $s\in[0,\infty)$, denote by $K_s(h)$  the convex set
 $$
K_s(h):=\{(x,t)\in L: h(x,t)\geq s\}
$$
and let $I_h:[0,\infty)\to[0,\infty)$ be the function given by
$$
I_h(s):=\int_{K_s(h)}e^{-t}dxdt.
$$
Notice that $I_h(s)$ is continuous, non-increasing and, by Lemma \ref{logconcave}, log-concave. Besides,
\begin{eqnarray*}
I_{h}(0)&=&\int_L e^{-t}dxdt=\int_0^\infty e^{-t}\vol_n\{x\in\R^n:f(x)\geq e^{-t}\Vert f\Vert_\infty\}dt\cr
&=&\int_0^1\vol_n\{x\in\R^n:f(x)\geq t\Vert f\Vert_\infty\})dt=\int_{\R^n}\frac{f(x)}{\Vert f\Vert_\infty}dx.\cr
\end{eqnarray*}
Notice that, from the definition of $\Phi_\gamma(h)$,
\begin{eqnarray*}
\Phi_\gamma(h)^\gamma&=&\frac{1}{\Gamma\left(1+\gamma\right)\int_L e^{-t}dxdt}\int_L h^\gamma(x,t)e^{-t}dxdt\cr
&=&\frac{1}{\Gamma\left(1+\gamma\right)\int_L e^{-t}dxdt}\int_L \int_0^{h(x,t)}\gamma s^{\gamma-1}e^{-t}dsdxdt\cr
&=&\frac{1}{\Gamma\left(1+\gamma\right)\int_L e^{-t}dxdt}\int_0^\infty\int_{K_s(h)} \gamma s^{\gamma-1}e^{-t}dxdtds\cr
&=&\frac{1}{\Gamma\left(1+\gamma\right)\int_L e^{-t}dxdt}\int_0^\infty I_h(s) \gamma s^{\gamma-1}ds.\cr
\end{eqnarray*}
Let $\bar{h}:L\to[0,\infty)$ be the function
$$
\bar{h}(x,t):=\bar{m}\left(t+\log\frac{f(x)}{\Vert f\Vert_\infty}\right),
$$
with $\bar{m}$ some constant to be determined later. Notice that $\bar{h}$ is non-negative, since for every $(x,t)\in L$ $f(x)\geq e^{-t}\Vert f\Vert_\infty$. Besides, $\bar{h}$ is concave on $L$,
\begin{eqnarray*}
K_s(\bar{h})&=&\{(x,t)\in L :\bar{m}(t+\log\frac{f(x)}{\Vert f\Vert_\infty})\geq s\}\cr
&=&\{(x,t)\in L :f(x)\geq e^{-t}e^{\frac{s}{\bar{m}}}\Vert f\Vert_\infty\},\cr
\end{eqnarray*}
and
\begin{eqnarray*}
I_{\bar{h}}(s)&=&\int_{K_s(\bar{h})}e^{-t}dxdt\cr
&=&\int_0^\infty e^{-t}\vol_n(\{x\in\R^n: (x,t)\in L,f(x)\geq e^{-t}e^{\frac{s}{\bar{m}}}\Vert f\Vert_\infty\})dt\cr
&=&\int_0^1 \vol_n(\{x\in\R^n: (x,-\log t)\in L, f(x)\geq t e^{\frac{s}{\bar{m}}}\Vert f\Vert_\infty\})dt\cr
&=&e^{-\frac{s}{\bar{m}}}\int_0^{e^{\frac{s}{\bar{m}}}} \vol_n(\{x\in\R^n:(x,\frac{t}{\bar{m}}-\log t)\in L, f(x)\geq t\Vert f\Vert_\infty\})dt\cr
&=&e^{-\frac{s}{\bar{m}}}\int_{\R^n}\frac{f(x)}{\Vert f\Vert_\infty}dx= e^{-\frac{s}{\bar{m}}}I_h(0).
\end{eqnarray*}
Consequently, for any $\gamma\in(0,\infty)$
\begin{eqnarray*}
\Phi_\gamma(\bar{h})^\gamma&=&\frac{1}{\Gamma\left(1+\gamma\right)}\int_0^\infty e^{-\frac{s}{\bar{m}}}\gamma s^{\gamma-1}ds\cr
&=&\frac{\bar{m}^\gamma}{\Gamma\left(1+\gamma\right)}\int_0^\infty e^{-s}\gamma s^{\gamma-1}ds\cr
&=&\bar{m}^\gamma.
\end{eqnarray*}
Let now $0<\gamma_1<\gamma_2$ and take $\bar{m}:=\Phi_{\gamma_1}(h)$. We have that $\Phi_{\gamma_1}(h)=\Phi_{\gamma_1}(\bar{h})$ and therefore
\begin{equation}\label{eq:IntegralEqualsZero}
\int_0^\infty (I_h(s)-I_{\bar{h}}(s))\gamma_1s^{\gamma_1-1}ds=0.
\end{equation}
Since $-\log I_h(s)$ is convex, non-decreasing, $\lim_{s\to\infty} -\log I_h(s)=\infty$, $-\log I_{\bar{h}}(s)=\frac{s}{\bar{m}}-\log I_h(0)$ is an affine function, and $-\log I_h(0)=-\log I_{\bar{h}}(0)$, then $I_h(s)$ and $I_{\bar h}(s)$ switch at most in one point. Thus, there exists $s_0\in[0,\infty)$ such that $I_{h}(s)\geq I_{\bar{h}}(s)$ if $0<s\leq s_0$ and $I_{h}(s)\leq I_{\bar{h}}(s)$ if $s\geq s_0$. Consequently, from \eqref{eq:IntegralEqualsZero} we have that
$$
\int_0^{s_0} (I_h(s)-I_{\bar{h}}(s))\gamma_1s^{\gamma_1-1}ds=\int_{s_0}^\infty (I_{\bar{h}}(s)-I_h(s))\gamma_1s^{\gamma_1-1}ds.
$$
Now, we have that for any $\gamma_2\geq\gamma_1$ that
\begin{eqnarray*}
&&\Phi_{\gamma_2}(h)^{\gamma_2}-\Phi_{\gamma_2}(\bar{h})^{\gamma_2}=\frac{1}{\Gamma\left(1+\gamma\right)\int_L e^{-t}dxdt}\int_0^\infty (I_h(s)-I_{\bar{h}}(s)) \gamma_2 s^{\gamma_2-1}ds\cr
&=&\frac{1}{\Gamma\left(1+\gamma\right)\int_L e^{-t}dxdt}\left(\int_0^{s_0} (I_h(s)-I_{\bar{h}}(s)) \gamma_2 s^{\gamma_2-1}ds-\int_{s_0}^\infty (I_{\bar{h}}(s)-I_{h}(s)) \gamma_2 s^{\gamma_2-1}ds\right)\cr
&=&\frac{\gamma_2}{\gamma_1\Gamma\left(1+\gamma\right)\int_L e^{-t}dxdt}\left(\int_0^{s_0} (I_h(s)-I_{\bar{h}}(s)) \gamma_1 s^{\gamma_1-1}s^{\gamma_2-\gamma_1}ds\right.\cr
&-&\left.\int_{s_0}^\infty (I_{\bar{h}}(s)-I_{h}(s)) \gamma_1 s^{\gamma_1-1}s^{\gamma_2-\gamma_1}ds\right)\cr
&\leq&\frac{\gamma_2s_0^{\gamma_2-\gamma_1}}{\gamma_1\Gamma\left(1+\gamma\right)\int_L e^{-t}dxdt}\left(\int_0^{s_0} (I_h(s)-I_{\bar{h}}(s)) \gamma_1 s^{\gamma_1-1}ds-\int_{s_0}^\infty (I_{\bar{h}}(s)-I_{h}(s)) \gamma_1 s^{\gamma_1-1}ds\right)\cr
&=&0.
\end{eqnarray*}
Therefore, $\Phi_{\gamma_2}(h)\leq \Phi_{\gamma_2}(\bar{h})=\Phi_{\gamma_1}(\bar{h})=\Phi_{\gamma_1}(h)$ and we obtain the result.
\end{proof}

The following lemma is well known  (see \cite{rogers}).

\begin{lemma}
Let $a_1,\dots a_m>0$, $b_1,\dots b_m>0$. Then
$$
(b_1^{a_1}b_2^{a_2}\cdots b_m^{a_m})^{\frac{1}{a_1+\cdots +a_m}}
\le
\frac{a_1b_1+a_2b_2+\cdots a_mb_m}{a_1+a_2+\cdots+a_m},
$$
with equality for $m>1$ if and only if $b_1=\cdots=b_m$.

As a consequence, let $\alpha_1,\dots\alpha_m>0$, $\beta_1,\dots \beta_m>0$ and $\sigma={\alpha_1+\alpha_2+\cdots+\alpha_m}$. Then
\begin{equation}\label{alfabetanum}
\beta_1^{\alpha_1}\beta_2^{\alpha_2}\cdots \beta_m^{\alpha_m}
\le
\frac{\alpha_1\beta_1^\sigma+\alpha_2\beta_2^\sigma+\cdots \alpha_m\beta_m^\sigma}{\sigma}
\end{equation}
with equality for $m>1$ if and only if $\beta_1=\cdots=\beta_m$.
\end{lemma}

\begin{proof}
Replacing $a_i$ by $a\cdot a_i$ and $b_i$ by $b\cdot b_i$ with appropriate $a$ and $b$ (the inequality above does not change), we may assume that
 $\sum_{i=1}^m a_i=\sum_{i=1}^m a_ib_i=1$. Then the inequality becomes
 $$
 b_1^{a_1}b_2^{a_2}\cdots b_m^{a_m}\le1
 $$
and taking logarithm, $\sum_{i=1}^m a_i\log b_i\le0$. This can be obtained, under the normalization above, using the inequality $\log t\le t-1$, with equality if and only if $t=1$.

The second inequality can be obtained from the first one taking $a_i=\alpha_i$ and $b_i=\beta_i^{\sigma}$.
\end{proof}

\begin{lemma}\label{lem:ProductsIntegrals}
Let $f\in\F(\R^n)$ and let $L$ be the convex set $L=\{(x,t)\in\R^n\times [0,\infty):f(x)\geq e^{-t}\Vert f\Vert_\infty\}$. Let $h_1,\dots,h_m:L\to[0,\infty)$ be continuous, concave, not identically null functions, $\alpha_1,\dots,\alpha_m>0$ and $\sigma=\alpha_1+\dots+\alpha_m$. Then,
$$
\frac{1}{\int_L e^{-t}dxdt}\int_L \prod_{i=1}^mh_i^{\alpha_i}(x,t)e^{-t}dxdt\leq \prod_{i=1}^m \left(\frac{1}{\int_L e^{-t}dxdt}\int_L h_i^{\sigma}(x,t)e^{-t}dxdt\right)^\frac{\alpha_i}{\sigma}.
$$
\end{lemma}
\begin{proof}
Replacing each $h_i$ by $\lambda_i h_i$, for some $\lambda_i>0$, we can assume without loss of generality, that for every $1\leq i\leq m$
$$
\frac{1}{\int_L e^{-t}dxdt}\int_L h_i^{\sigma}(x,t)e^{-t}dxdt=1.
$$
Thus, we have to prove that
$$
\frac{1}{\int_L e^{-t}dxdt}\int_L \prod_{i=1}^mh_i^{\alpha_i}(x,t)e^{-t}dxdt\leq 1.
$$
For any fixed $(x,t)\in C$, apply (\ref{alfabetanum}) with $\beta_i=h_i(x,t)$ to obtain
$$
h_1(x,t)^{\alpha_1}
\cdot
h_2(x,t)^{\alpha_2}
\cdots
h_m(x,t)^{\alpha_m}
\le
\sum_{i=1}^m\frac{\alpha_i}{\sigma}h_i(x,t)^{\sigma}.
$$
Multiplying by $\frac{e^{-t}}{\int_L e^{-t}dxdt}$ and integrating over $L$ we obtain the result.
\end{proof}

\begin{proof}[Proof of Theorem \ref{thm:BerwaldExponential}]
Using Lemma \ref{lem:BerwaldExponentialDensity} with $\gamma_1=\alpha_i$, $\gamma_2=\sigma$, and $h=h_i$ we have that for every $1\leq i\leq m$
$$
\left(\frac{1}{\Gamma\left(1+\sigma\right)\int_L e^{-t}dxdt}\int_L h_i^{\sigma}(x,t)e^{-t}dxdt\right)^\frac{\alpha_i}{\sigma}\leq \frac{1}{\Gamma\left(1+\alpha_i\right)\int_L e^{-t}dxdt}\int_L h_i^{\alpha_i}(x,t)e^{-t}dxdt.
$$
Multiplying in $i=1,\dots,m$ we obtain
$$
\frac{1}{\Gamma\left(1+\sigma\right)}\prod_{i=1}^m\left(\frac{1}{\int_L e^{-t}dxdt}\int_L h_i^{\sigma}(x,t)e^{-t}dxdt\right)^\frac{\alpha_i}{\sigma}\leq \prod_{i=1}^m\frac{1}{\Gamma\left(1+\alpha_i\right)\int_L e^{-t}dxdt}\int_L h_i^{\alpha_i}(x,t)e^{-t}dxdt.
$$
Using Lemma \ref{lem:ProductsIntegrals} we obtain the result.
\end{proof}

\section{Restricted Loomis-Whitney type inequalities}\label{sec:LW}

In this section we prove Theorems \ref{th:localLoomisWhitneyReverse} and \ref{th:localLoomisWhitney}. Their proofs are found in their own subsections. Following the idea developed in the previous sections, we also prove their functional counterparts, which are Theorems \ref{th:locLWFunctProjSect} and \ref{thm: FunctionalLW}. Before this, let us first recall what are the classical direct and reverse Loomis-Whitney inequalities, their previously known local versions, and connect them with our inequalities.

The classical Loomis-Whitney inequality \cite{LW} states that for any $K\in\K^n$ and $H_k\in\L^n_{n-1}$, $k=1,\dots,n$, with $H_k^{\bot}\subset H_l$ for $l\neq k$, then
\begin{equation}\label{eq:LoomisWhitney}
\vol_n(K)^{n-1}\leq\prod_{k=1}^n\vol_{n-1}(P_{H_k}K).
\end{equation}
Equality holds above if and only if $K$ is a box with facets parallel to each $H_k$. A reverse inequality, in which
projections are replaced by sections, was proved by Meyer \cite{Me}. Namely, under the same assumptions on the hyperplanes $H_k$,
\begin{equation}\label{eq:reverseLoomisWhitney}
\frac{n!}{n^n}\prod_{k=1}^{n}\vol_{n-1}(K\cap H_k)\leq \vol_{n}(K)^{n-1}.
\end{equation}
Moreover, equality holds if and only if $K$ is a crosspolytope whose generating vectors are orthogonal to each of the subspaces $H_k$. See also \cite{CGG,KSZ} for a reverse Loomis-Whitney inequality via projections.
If, rather than considering $n$ subspaces, we restrict to 2 of them, then we arrive onto \emph{local Loomis-Whitney type} inequalities. An exhaustive study of those inequalities is done in \cite{BGL}.

Our reverse local Loomis-Whitney inequality in Theorem \ref{th:localLoomisWhitneyReverse}, as well as the next Lemma \ref{RSwithIntersection},
are results of the type \cite[Theorem 1.3]{BGL}, which was already a generalization of Bollob\'as and Thomason \cite{BoTh}. Here, we solve completely the case (in the notation of \cite{BGL}) of $s=1$ and $r=2$. Indeed, our result improves a factor of the form $\frac{(n-j)^{n-j}(n-i)^{n-i}}{(c_0(2n-i-j))^{2n-i-j}}$ (a constant like the one of Remark \ref{rmk:worsebound}) by the factor ${2n-i-j\choose n-i}^{-1}$ which is sharp.

Our local Loomis-Whitney inequality Theorem \ref{th:localLoomisWhitney} is a result of the type \cite[Theorem 1.2]{BGL}. We quote only those results pertaining to two subspaces overlapping.
For any $K\in\K^n$ and $H_1,H_2\in\mathcal{L}^n_{n-1}$  Giannopoulos, Hartzoulaki, and  Paouris \cite[Lemma 4.1]{GiHaPa} showed that
\begin{equation}\label{eq:GiHaPa1}
\vol_{n-2}(P_{H_1\cap H_2}K)\vol_n(K) \leq \frac{2(n-1)}{n} \vol_{n-1}(P_{H_1}K)  \vol_{n-1}(P_{H_2}K)
\end{equation}
whereas for any $E\in\mathcal{L}^n_i$, $H\in\mathcal{L}^n_j$, $i+j\geq n$, $E^\bot\subset H$, the authors in \cite[Theorem 1.2]{BGL} extended this onto
\[
\vol_{i+j-n}(P_{E\cap H}K) \vol_n(K) \leq \gamma(n,2n-i-j,1,2)^{-1}\vol_{i}(P_EK)\vol_{j}(P_HK),
\]
where
$\gamma(n,2n-i-j,1,2)={{n\choose 2n-i-j}}/{{(i+j)/2\choose i+j-n}^2}$
(cf.~also \cite[Theorem 5.4]{SoZv}).
Theorem \ref{th:localLoomisWhitney} improves the previous inequalities by obtaining sharp estimates for any choice of $i$ and $j$. In particular, if $i=j$ we obtain the same result, and since the binomial coefficients are concave, when $i\neq j$ our estimate gives a better (and best possible) constant.

\subsection{Reverse (local) Loomis-whitney inequalities}\label{subsec:LWreverse}

To prove Theorem \ref{th:localLoomisWhitneyReverse}, we shall first prove the following lemma
and then use the symmetrization procedure described in Section \ref{sec:Estimates by marginals}.

\begin{lemma}\label{RSwithIntersection}\label{th:localRevLW}
    Let $K\in\K^n$, $E\in\L^n_i$, and $H\in\L^n_j$ be such that $i,j\in\{2,\dots,n-1\}$, $i+j\geq n+1$,
    and $E^{\bot}\subset H$. Let $k:=i+j-n$, so that $1\leq k\leq n-2$.
    Then
	\begin{equation}\label{eq:revLWandRS}
\vol_{i}(K\cap E)\vol_{j}(K\cap H)\le {n-k\choose n-i}
\max_{x\in \R^n}(\vol_{k}(K\cap(x+E\cap H))) \vol_n(K).  %
\end{equation}
	Equality holds if and only if there exist $K_1\subset E^\bot$, $K_2\subset H^\bot$, $K_3\subset E\cap H$, such that
	$P_{(E\cap H)^\bot}K$ is a translate of $\conv(\{K_1, K_2\})$, and
	for every $x\in (E\cap H)^\bot$ the convex set $K\cap(x+E\cap H)$ is a translate of $K_3$.
\end{lemma}

\begin{proof}
Letting $F = E\cap H$, we have that $\dim (F) = j+i-n =k$, and $\R^n = F\oplus E^\bot\oplus  H^\bot$. 
Let us consider
\[
f:F^{\bot}\rightarrow[0,\infty],\quad f(x):= \vol_{k}(K\cap(x+F)).
\]
Brunn's concavity theorem implies that $f$ is $\frac{1}{k}$-concave, and hence, in particular, log-concave.
We can thus apply Theorem \ref{th:splittRS} to the two orthogonal subspaces spanning $F^\bot$, which are $E^\perp$ and $H^\perp$.
We get that
\[ 	\Vert f\Vert_{\infty} \int_{F^\bot}f(z)dz
\ge {n-k\choose n-j}^{-1} \int_{H^\bot}P_{H^\perp}f(x)dx \int_{E^\bot }f(y)dy.  \]
Note that
\begin{eqnarray*}
\|f\|_\infty &=& \max_{x\in F^{\bot}}\vol_{k}(K\cap(x+F)),\qquad
P_{H^{\bot}}f(x)  =  \max_{y\in E^\bot} \vol_k(K\cap(x+y+F)),\\
\int_{E^\bot }f(y)dy &=& 
\vol_{j}(K\cap H),\qquad {\rm and}  \qquad
\int_{F^\bot}f(z)dz   =   \vol_n(K).
\end{eqnarray*}
 Therefore, our inequality reads
\[ 	\max_{z\in \R^n}\vol_{k}(K\cap(z+F)) \vol_n(K)
\ge {n-k\choose n-j}^{-1} \int_{H^\bot} \max_{y\in E^\bot} \vol_k(K\cap(x+y+F)) dx \vol_{j}(K\cap H).  \]
Finally, we use the inequality, for 
$0\in E^\bot$,
\[ \int_{H^\bot} \max_{y\in E^\bot} \vol_k(K\cap(x+y+F)) dx \ge \int_{H^\bot}   \vol_k(K\cap(x+F))  = \vol_i(K\cap E),\]
which plugging back into the inequality gives our claim:
\[ 	\max_{z\in \R^n}\vol_{k}(K\cap(z+F)) \vol_n(K)
\ge {n-k\choose n-j}^{-1}  \vol_i(K\cap E)   \vol_{j}(K\cap H).  \]
The equality case is treated in Subsection \ref{subsec:Equality in RSwithprojandintersection}.
\end{proof}

Once we have an inequality with respect to sections, we can apply it to a symmetrization of a given function, and get an inequality involving projections.

\begin{proof}[Proof of Theorem \ref{th:localLoomisWhitneyReverse}]
Given a body $K$, denote $\widetilde K:=S_EK$, so that $\vol_n(K) = \vol_n(\widetilde K)$. Denote as before $F = E\cap H$ and $\dim F = k = i+j-n$.
Apply Lemma \ref{RSwithIntersection} to $\widetilde K$ to get that
\[ 	\max_{x\in \R^n}\vol_{k}(\widetilde K\cap(x+F)) \vol_n(K)
\ge {n-k\choose n-j}^{-1}  \vol_i(\widetilde K \cap E)  \vol_{j}(\widetilde K \cap H).  \]
Clearly, $\widetilde K \cap E = P_E K$ by the definition of the symmetrization of a body. On the other hand, as $E^\perp \subset H$ and the volumes along fibers $x+E^\perp$ are preserved by symmetrization, so is the volume along $H$, and we have that
\[ \vol_{j}(\widetilde K \cap H) = \vol_{j}(K \cap H).\]
Finally, as $F\subset E$ we see that
  \[ \max_{x\in \R^n} \vol_k ( \widetilde K\cap(x+F))  =  \max_{x\in F^\bot\cap E} \vol_k ( P_E K\cap(x+F))  \]
  and $F^\bot\cap E = H^\bot$. Thus, our main inequality can be written as
  \[ 	\max_{x\in \R^n}\vol_{k}(P_E K\cap(x+F)) \vol_n(K)
  \ge {n-k\choose n-j}^{-1}  \vol_i(P_EK) \vol_{j}(K \cap H),   \]
  which is the statement of Theorem \ref{th:localLoomisWhitneyReverse}, after we note that $P_E K\cap(x+F) = P_E(K\cap H)$.

  The equality case is treated in Subsection \ref{subsec:Equality in RSwithprojandintersection}.
\end{proof}

\begin{rmk}
Let us observe that the Hanner polytope
\[
K_{n,i,j}:=\underset{k=1,\dots,2n-i-j}{\conv}(\pm e_k)+\sum^n_{k=2n-i-j+1}[-e_k,e_k]
\]
together with $E:=\lin\{e_1,\dots,e_{n-j},e_{2n-i-j+1},\dots,e_n\}$ and $H:=\lin\{e_{n-j+1},\dots,e_n\}$, whenever $i+j\geq n+1$,
attains equality in Theorem \ref{th:localLoomisWhitneyReverse} and Lemma \ref{th:localRevLW}.
\end{rmk}

The next result is a functional version of Lemma \ref{RSwithIntersection}.
\begin{lemma}\label{lem:locLWFunctRev}
Let $f\in\mathcal{F}(\R^n)$ and let $H\in\mathcal{L}^n_i$ and $E\in\mathcal{L}^n_j$ be such that $i,j\in\{2,\dots,n-1\}$, $i+j\geq n+1$, and $E^{\bot}\subset H$. Let $k:=i+j-n$, and hence $1\leq k\leq n-2$. Then
\[
\int_{E}f(x)dx\,\int_{H}f(y)dy \leq {n-k\choose n-i}\max_{x\in(E\cap H)^\bot}\int_{x+E\cap H}f(w)dw\int_{\R^n}f(z)dz.
\]
\end{lemma}

\begin{proof}
Let us define the function
\[
F:(E\cap H)^\bot\rightarrow[0,\infty),\quad F(x):=\int_{x+E\cap H}f(y)dy.
\]
By Pr\'ekopa-Leindler inequality, $F$ is a log-concave function.
Theorem \ref{th:splittRS} implies that
\[
\begin{split}
{n-k\choose n-i} \max_{x\in(E\cap H)^\bot} & \int_{x+E\cap H}f(w)dw\,\int_{\R^n}f(z)dz
={n-k\choose n-i}\Vert F\Vert_{\infty}\int_{(E\cap H)^\bot}F(z)dz\\
&\geq\int_{H^\bot}P_{H^\bot}F(x)dx\,\int_{E^\bot}F(y)dy\\
&=\int_{H^\bot}\max_{x_0\in E^\bot}F(x+x_0)dx\int_{E^\bot}F(y)dy\\
&\geq\int_{H^\bot}F(x)dx\int_{E^\bot}F(y)dy\\
&=\int_{H^\bot}\int_{x+E\cap H}f(\bar{x})d\bar{x}dx\int_{E^\bot}\int_{x+E\cap H}f(\bar{y})d\bar{y}dy\\
&=\int_{E}f(x)dx\,\int_{H}f(y)dy.
\end{split}
\]
\end{proof}

\begin{proof}[Proof of Theorem \ref{th:locLWFunctProjSect}]
Let us define the symmetral of $f$ with respect to $E$ by $\tilde{f}:=S_Ef$. Using \eqref{eq:projequalssymm}, \eqref{eq:integsymmsubspace},
and \eqref{eq:integsymm} we have that
\[
\tilde{f}(x)=P_Ef(x)\quad\forall \,x\in E, \quad \int_H\tilde{f}(y)dy=\int_Hf(y)dy,\quad\text{and}\quad \int_{\R^n}\tilde{f}(z)dz=\int_{\R^n}f(z)dz.
\]
Moreover, since $\tilde{f}(x)\geq\tilde{f}(x+y)=\tilde{f}(x-y)$ for every $(x,y)\in E\times E^\bot$, then
\[
\max_{x_0\in(E\cap H)^\bot}\int_{x+E\cap H}\tilde{f}(w)dw=\max_{x_0\in H^\bot}\int_{x+E\cap H}\tilde{f}(w)dw=\max_{x\in H^\bot}\int_{x+E\cap H}P_Ef(w)dw.
\]
All this together with Lemma \ref{lem:locLWFunctRev} imply that
\[
\begin{split}
\int_EP_Ef(x)dx\int_Hf(y)dy =\int_E\tilde{f}(x)dx\int_H\tilde{f}(y)dy & \leq {n-k\choose n-i}\max_{x_0\in (E\cap H)^\bot}\int_{x+E\cap H}\tilde{f}(w)dw\int_{\R^n}\tilde{f}(z)dz\\
& ={n-k\choose n-i}\max_{x_0\in H^\bot}\int_{x+E\cap H}P_Ef(w)dw\int_{\R^n}f(z)dz.
\end{split}
\]
\end{proof}

\subsection{Direct (local) Loomis-Whitney inequalities}\label{subsec:LWdirect}

We start proving Theorem \ref{th:localLoomisWhitney}. We follow the ideas of
Giannopoulos et.~al.~in \cite[Lemma 4.1]{GiHaPa},  making use of a classic result by Berwald \cite[Satz 8]{Ber}. For the sake of completeness and since we were not able to find an English translation of this result, we will write a complete proof of it, together with its equality cases (cf.~Appendix \ref{sec:appendixBerwald}, Theorem \ref{berwald}).

\begin{proof}[Proof of Theorem \ref{th:localLoomisWhitney}]
On the one hand, observe that for every $x\in P_{E\cap H}K$
\begin{equation}\label{LocalLWeq1}
\begin{split}
&\vol_{n-k}(K\cap(x+(E\cap H)^\bot))\\
&\leq \vol_{n-j}(P_{x+H^\bot}(K\cap(x+(E\cap H)^\bot)))\,\vol_{n-i}(P_{x+E^\bot}(K\cap(x+(E\cap H)^\bot)))\\
&=\vol_{n-j}((P_EK)\cap(x+H^\bot))\,\vol_{n-i}((P_HK)\cap(x+E^\bot)).
\end{split}
\end{equation}
Hence
\begin{equation*}\label{LocalLWeq2}
\begin{split}
\vol_n(K) & =\int_{P_{E\cap H}K}\vol_{n-k}(K\cap(x+(E\cap H)^{\bot}))dx \\
& \leq \int_{P_{E\cap H}K}\vol_{n-j}((P_EK)\cap(x+H^{\bot}))\vol_{n-i}((P_HK)\cap(x+E^{\bot}))dx \\
& =\int_{P_{E\cap H}K}\left(\vol_{n-j}((P_EK)\cap(x+H^{\bot}))^{\frac{1}{n-j}}\right)^{n-j}\left(\vol_{n-i}((P_HK)\cap(x+E^{\bot}))^{\frac{1}{n-i}}\right)^{n-i}dx.
\end{split}
\end{equation*}
Let us define $f_1,f_2:P_{E\cap H}K\rightarrow[0,\infty)$ by
\[
f_1(x):=\vol_{n-j}((P_EK)\cap(x+H^{\bot}))^{\frac{1}{n-j}}\quad \text{and} \quad f_2(x):=\vol_{n-i}((P_HK)\cap(x+E^{\bot}))^{\frac{1}{n-i}},
\]
respectively. By Brunn-Minkowski theorem, $f_1$ and $f_2$ are concave functions. Berwald's Theorem (which is stated and proven in our appendix as Theorem \ref{berwald})
applied to $f_1,f_2$ with $\alpha_1=n-j$ and $\alpha_2=n-i$ (recalling that we set $k=i+j-n$) implies that
\begin{equation}\label{LocalLWeq2}
\begin{split}
\frac{1}{\vol_k(P_{E\cap H}K)} & \int_{P_{E\cap H}K}f_1(x)^{n-j}f_2(x)^{n-i}dx \\
&\leq \frac{{i\choose k}{j\choose k}}{{n\choose k}} \frac{1}{\vol_k(P_{E\cap H}K)^2}\int_{P_{E\cap H}K}f_1(x)^{n-j}dx\int_{P_{E\cap H}K}f_2(x)^{n-i}dx \\
& =\frac{{i\choose k}{j\choose k}}{{n\choose k}}\frac{1}{\vol_k(P_{E\cap H}K)^2} \vol_i(P_EK) \vol_j(P_HK).
\end{split}
\end{equation}
The above inequality together with \eqref{LocalLWeq2} implies
\[
\vol_n(K)\leq \frac{{i\choose k}{j\choose k}}{{n\choose k}} \frac{1}{\vol_k(P_{E\cap H}K)}\vol_i(P_EK)\vol_j(P_HK),
\]
which shows \eqref{dirLocLW}.
The equality case is treated in Section \ref{subsec:equality in directLW}.
\end{proof}

\begin{rmk}
Let us observe that the Hanner polytope
\[
K_{n,i,j}:=\conv\left\{\underset{k=1,\dots,n-j}{\conv}(\pm e_k)+\underset{k=n-j+1,\dots,2n-i-j}{\conv}(\pm e_k),\sum^n_{k=2n-i-j+1}[-e_k,e_k]\right\}
\]
together with $E:=\lin\{e_1,\dots,e_{n-j},e_{2n-i-j+1},\dots,e_n\}$ and $H:=\lin\{e_{n-j+1},\dots,e_n\}$, whenever $i+j\geq n+1$,
attains equality in Theorem \ref{th:localLoomisWhitney}.
\end{rmk}

Actually, we can prove the following more general result. It goes in the direction of \cite{BGL}  involving more than 2 subspaces into the game, but with some extra conditions.

\begin{thm}
Let $K\in\K^n$, $E_j\in\mathcal{L}^n_{i_j}$, $j=1,\dots,m$, such that $\R^n=E_1 \oplus \cdots \oplus E_m$. Then
\[
\vol_{i_1}(P_{E_1}K)^{m-2}\vol_n(K)\leq\frac{\prod_{j=2}^m{i_1+i_j\choose i_1}}{{n\choose i_1}}\prod_{j=2}^m\vol_{i_j}(P_{\lin(\{E_1,E_j\})K}).
\]
\end{thm}

\begin{rmk}
Theorem \ref{th:localLoomisWhitney} follows from the previous theorem choosing $m:=3$, $E_1:=E\cap H$, $E_2:=E^\bot$, and $E_3:=H^\bot$.
\end{rmk}

In order to conclude this section, we prove the functional version of the local Loomis-Whitney inequality.

\begin{proof}[Proof of Theorem \ref{thm: FunctionalLW}]
Let  $C$ be the set
$$
C=\{(x,t)\in\R^{n}\times[0,\infty): f(x)\geq e^{-t}\Vert f\Vert_\infty\}.
$$
Since $f$ is log-concave, then $C$ is convex. Besides
\begin{eqnarray*}
\int_C e^{-t}dxdt&=&\int_0^\infty e^{-t}\vol_n(\{x\in\R^n:f(x)\geq e^{-t}\Vert f\Vert_\infty\})dt\cr
&=&\int_0^1\vol_n(\{x\in\R^n:f(x)\geq s\Vert f\Vert_\infty\})ds\cr
&=&\int_{\R^n}\frac{f(x)}{\Vert f\Vert_\infty}dx.
\end{eqnarray*}
For any linear subspace $F\in\mathcal{L}^n_{l}$ let us call $\overline{F}=\textrm{span}\{F, e_{n+1}\}$ and notice that $\overline{E\cap H}=\overline{E}\cap\overline{H}$. Notice also that
\begin{eqnarray*}
\int_{P_{\overline{F}}C} e^{-t}dxdt&=&\int_0^\infty e^{-t}\vol_l(\{x\in F:(x,t)\in P_{\overline{F}}C\})dt\cr
&=&\int_0^\infty e^{-t}\vol_l(\{x\in F:\sup_{y\in F^\perp}f(x+y)\geq e^{-t}\Vert f\Vert_\infty\})dt\cr
&=&\int_0^1 \vol_l(\{x\in F:P_F f(x)\geq s\Vert f\Vert_\infty\})ds\cr
&=&\int_{F}\frac{P_Ff(x)}{\Vert f\Vert_\infty}dx.\cr
\end{eqnarray*}
Now, notice that
\begin{eqnarray*}
&&\int_C e^{-t}dxdt=\int_{P_{\overline{E\cap H}}C}e^{-t}\vol_{n-k}(C\cap(x,t)+(\overline{E\cap H})^\perp)dt\cr
&\leq&\int_{P_{\overline{E\cap H}}C}e^{-t}\vol_{n-i}(P_{\overline{E}^\perp}(C\cap((x,t)+(\overline{E\cap H})^\perp)))\vol_{n-j}(P_{\overline{H}^\perp}(C\cap((x,t)+(\overline{E\cap H})^\perp)))dt\cr
&=&\int_{P_{\overline{E\cap H}}C}e^{-t}\vol_{n-i}(P_{\overline{H}}C\cap((x,t)+\overline{E}^\perp))\vol_{n-j}(P_{\overline{E}}C\cap((x,t)+\overline{H}^\perp))dt\cr
&=&\int_{P_{\overline{E\cap H}}C}e^{-t}\left(\vol_{n-i}(P_{\overline{H}}C\cap((x,t)+\overline{E}^\perp))^\frac{1}{n-i}\right)^{n-i}
\left(\vol_{n-j}(P_{\overline{E}}C\cap((x,t)+\overline{H}^\perp))^\frac{1}{n-j}\right)^{n-j}dt\cr
\end{eqnarray*}
By Brunn-Minkowski inequality, the functions $f_1(x,t):=\vol_{n-i}(P_{\overline{H}}C\cap((x,t)+\overline{E}^\perp))^\frac{1}{n-i}$ and $f_2(x,t):=\vol_{n-j}(P_{\overline{E}}C\cap((x,t)+\overline{H}^\perp))^\frac{1}{n-j}$ are concave, and then by Theorem \ref{thm:BerwaldExponential} applied to $L=P_{\overline{E\cap H}}C=\{(x,t):P_{E\cap H}f(x)\geq e^{-t}\Vert f\Vert_\infty\}$ we have that this quantity is bounded above by
\begin{eqnarray*}
&&{n-k\choose n-i}\frac{1}{\int_{P_{\overline{E\cap H}}C} e^{-t}dxdt}\int_{P_{\overline{E\cap H}}C}e^{-t}\vol_{n-i}(P_{\overline{H}}C\cap((x,t)+\overline{E}^\perp))dxdt\cr
&&\times\int_{P_{\overline{E\cap H}}C}e^{-t}\vol_{n-j}(P_{\overline{E}}C\cap((x,t)+\overline{H}^\perp))dxdt\cr
&&={n-k\choose n-i}\frac{1}{\int_{P_{\overline{E\cap H}}C} e^{-t}dxdt}\int_{P_{\overline{H}}C}e^{-t}dxdt\int_{P_{\overline{E}}C}e^{-t}dxdt.\cr
\end{eqnarray*}
\end{proof}

\section{Rogers-Shephard type inequalities}\label{sec:Rogers-Shephard type inequalities}

In this section we derive various functional Rogers-Shephard type inequalities. We obtain most of them via Theorem \ref{th:splittRS} or Theorem \ref{SplittingColesanti}, following some of the ideas used by Rogers and Shephard in \cite{RS58}.

\begin{lemma}\label{lem:logConcF}
Let $f,g\in\F(\R^n)$. Then for the function
\[
F:\R^{2n}\times [0,1]\to[0,\infty),\quad F(z_1,z_2,t):=f\left(\frac{z_1}{t}\right)^tg\left(\frac{z_2-z_1}{1-t}\right)^{1-t},
\]
 which for $t=0$ is defined as $g(z_2-z_1)$and for $t=1$ is defined as $f(z_1)$, we have that $F\in\F(\R^{2n+1})$, i.e., $F$ is a log-concave function.
\end{lemma}

\begin{proof}
The logarithm of $F$ is
$$
\log F(z_1,z_2,t)=t\log f\left(\frac{z_1}{t}\right)+(1-t)\log g\left(\frac{z_2-z_1}{1-t}\right).
$$
The function $t\log f\left(\frac{z_1}{t}\right)$ is a concave function, since $\log f$ is a concave function and then for any $0\leq\lambda\leq1$ and any $(z_1,z_2,t),(\overline{z}_1,\overline{z}_2,\overline{t})\in\R^{2n+1}$ we have
\begin{eqnarray*}
&&\left((1-\lambda)t+\lambda\overline{t}\right)\log f\left(\frac{(1-\lambda)z_1+\lambda\overline{z}_1}{(1-\lambda)t+\lambda\overline{t}}\right)=\\
&&\left((1-\lambda)t+\lambda\overline{t}\right)\log f\left(\frac{(1-\lambda)t}{(1-\lambda)t+\lambda\overline{t}}\frac{z_1}{t}+\frac{\lambda\overline{t}}{(1-\lambda)t+\lambda\overline{t}}\frac{\overline{z}_1}{\overline{t}}\right)\geq\\
&&(1-\lambda)t\log f\left(\frac{z_1}{t}\right)+\lambda\overline{t}\log f\left(\frac{\overline{z}_1}{\overline{t}}\right).
\end{eqnarray*}
In the same way, the function $(1-t)\log g\left(\frac{z_2-z_1}{1-t}\right)$ is concave and thus $F$ is log-concave.
\end{proof}

\begin{proof}[Proof of Theorem \ref{thm:RSFuncconvexhullbest}]
Define
\[  \Psi (z_1, z_2, t):\R^n\times\R^n\times (0,1)\to[0,\infty),\qquad  \Psi (z_1,z_2,t):=f\left(\frac{z_1}{1-t}\right)^{1-t}g\left(\frac{z_2}{t}\right)^t.
\]
Like in the previous lemma $\Psi$ is an integrable log-concave function.
Note that
\[ \int_{\R^{2n+1} } \Psi(z_1, z_2, t)dtdz_2dz_1 = \int_0^1 \left(t^n (1-t)^n \int_{\R^n} f(x)^{1-t}dx \int_{\R^n} g(y)^tdy\right) dt.
\]

Letting $H:= e_{2n+1}^\perp\in\mathcal L^{2n+1}_{2n}$,
$F(z_1,z_2) = f(z_1)\chi_{\{0\}}(z_2)$ and $G(z_1,z_2) = g(z_2)\chi_{\{0\}}(z_1)$, then
\[  P_{H} \Psi (z_1, z_2) = \sup_{0\le t\le 1}\Psi (z_1, z_2, t) = F\tilde{\star} G (z_1, z_2).\]
Therefore, by Theorem \ref{th:splittRS} for the subspace of dimension $2n$ of $\R^{2n+1}$ we know that
\begin{equation}\label{eq:psec0}
 \int_{H} P_{H} \Psi (z_1, z_2)dz_2dz_1 \int_0^1 \Psi (0,0,t)dt \le {2n+1\choose 2n} \Vert \Psi\Vert_{\infty}\int_{\R^{2n+1} } \Psi(z_1,z_2,t)dtdz_2dz_1.
\end{equation}
Since $\Vert \Psi\Vert_\infty=\max\{\Vert f\Vert_\infty,\Vert g\Vert_\infty\}$ and assuming that this maximum is $\Vert f\Vert_\infty$, then
\[
\begin{split}
\int_H F\tilde{\star} G (z_1, z_2) dz_1dz_2 &\frac{\Vert f\Vert_\infty-\Vert g\Vert_\infty}{\log\Vert f\Vert_\infty-\log\Vert g\Vert_\infty}  =
\int_H F\tilde{\star} G (z_1, z_2) dz_1dz_2 \int_0^1  f(0)^{1-t}g(0)^tdt \\
& \le (2n+1) \Vert f\Vert_\infty\int_0^1 \left(t^n (1-t)^n \int_{\R^n} f(x)^{1-t}dx \int_{\R^n} g(y)^tdy\right) dt.
\end{split}
\]

We now consider in $H=\R^{2n}$ the subspaces
$E = \{ (x,-x)\in H: x\in \R^n\}$ and $E^{\perp} = \{ (y,y)\in H: y\in \R^n\}$.
Again by Theorem \ref{th:splittRS} we have that
\begin{equation}\label{eq:psec} \int_{E} P_{E}(F\tilde{\star}G)(x)dx \int_{E^\perp} F\tilde{\star}G(y)dy\le \binom{2n}{n} \Vert F\tilde{\star} G\Vert_\infty \int_H F\tilde{\star}G(z)dz,
\end{equation}
to which we shall apply the upper bound which we obtained above.

Given some point $(x,x)$ in $E$, let us compute $(P_E F\tilde{\star}G)$. To this end, we first notice that
\[ P_E ( h_1\tilde{\star} h_2) = (P_E h_1)\tilde{\star} (P_E h_2).
\]
Indeed, this is best understood using the language of epi-graphs of the logarithms, and the fact that projections and convex hulls commute.
Therefore, we first compute $P_E F$ and $P_EG$:
\begin{eqnarray*}
(P_E F)(x,-x) & = & \sup_{y\in \R^n}    f( {x+y})\chi_{\{0\}} ({-x+y}) = f(2x) \\
(P_E G)(x,-x) & = & \sup_{y\in \R^n}    g({-x+y})\chi_{\{0\}}({x+y}) = g(-2x).
\end{eqnarray*}
Thus
\begin{eqnarray*}
(P_E( F\tilde{\star}G))(x,-x) & = & (	(P_E F)\tilde{\star}(P_E G) )(x,-x)  \\
	 & = & \sup_{0\le t\le 1}\sup_{x=(1-t)x_1+tx_2} f^{1-t}(2x_1) g^t(-2x_2)\\
	 & = & f\tilde{\star} g_{-} (2x)
\end{eqnarray*}
where $g_{-}(x) := g(-x)$.
Integrating over $E$ and taking into account that $E$ is a diagonal subspace, we see that
\[ \int_{E}  P_E( F\tilde{\star}G)(x)dx = 2^{n/2} \int_{\R^n} P_E( F\tilde{\star}G)(x,-x) dx = 2^{-n/2} \int_{\R^n}f\tilde{\star} g_{-}(z)dz. \]
Recall that
\[
F\tilde{\star}G (x,x) = \sup_{0\le t\le 1}\Psi (x, x, t) = \sup_{0\le t\le 1} f\left(\frac{x}{1-t}\right)^{1-t}g\left(\frac{x}{t}\right)^t
\]
where for $t=0$ in the supremum we mean just $f(x)$, and for $t=1$ just $f(y)$,
and as before we get an extra factor when integrating with respect to $x$ because the subspace is in fact diagonal:
\[
\int_{E^\perp} F\tilde{\star}G(x)dx =  2^{n/2} \int_{\R^n} F\tilde{\star}G(x,x)dx = 2^{n/2} \int_{\R^n}\sup_{0\le t\le 1} f\left(\frac{x}{1-t}\right)^{1-t}g\left(\frac{x}{t}\right)^t dx.
\]

Since $\Vert F\tilde{\star}G\Vert_\infty=\max\{\Vert f\Vert_\infty,\Vert g\Vert_\infty\}=\Vert f\Vert_\infty$, hence \eqref{eq:psec} reads as
\[
\int_{\R^n}f\tilde{\star} g_{-} \int_{\R^n}\sup_{0\le t\le 1} \left\{f\left(\frac{x}{1-t}\right)^{1-t}g\left(\frac{x}{t}\right)^t\right\} dx
\leq {2n\choose n} \Vert f\Vert_\infty \int_HF\tilde{\star}G
\]
and inserting \eqref{eq:psec0} onto \eqref{eq:psec} we thus obtain that
\begin{equation*}
\begin{split}
&\int_{\R^n}f\tilde{\star} g_{-}(z)dz  \int_{\R^n}\sup_{0\le t\le 1} \left\{f\left(\frac{x}{1-t}\right)^{1-t}g\left(\frac{x}{t}\right)^t\right\}dx  \\
& \le (2n+1){2n\choose n}\frac{\log\Vert f\Vert_\infty-\log\Vert g\Vert_\infty}{\Vert f\Vert_\infty-\Vert g\Vert_\infty}\Vert f\Vert_\infty^2
\int_0^1\left(t^n(1-t)^n\int_{\R^n} f(x)^{1-t}dx\int_{\R^n} g(y)^{t}dy\right)dt.
\end{split}
\end{equation*}

Changing $g$ by $g_{-}$ we obtain the statement of the theorem. The equality case is studied in Subsection \ref{subsub:equalRSFuncConvHullBest}.
\end{proof}

\begin{proof}[Proof of Theorem \ref{thm:RSFunc(n+1)}]
Let us define the function
\[
F:\R^n\times\R^n\times[0,1]\to[0,\infty),\quad F(z_1,z_2,t):=f\left(\frac{z_1}{t}\right)^tg\left(\frac{z_2-z_1}{1-t}\right)^{1-t},
\]
 which for $t=0$ is defined as $g(z_2-z_1)$ and for $t=1$ is defined as $f(z_1)$. By Lemma \ref{lem:logConcF}, $F$ fulfills $F\in\F(\R^{2n+1})$.
Then $\Vert F\Vert_{\infty}=\max\{\Vert f\Vert_{\infty},\Vert g\Vert_{\infty}\}$ and
\[
\int_{\R^{2n+1}}F(z)dz=\int_0^1t^n(1-t)^n\int_{\R^n}f(x)^tdx\int_{\R^n}g(y)^{1-t}dydt.
\]
Moreover, let $H:=\lin\{e_{n+1},\dots,e_{2n+1}\}\in\L^{2n+1}_{n+1}$, and
observe that
\[
P_HF(z_2,t)=\sup_{z_1}f\left(\frac{z_1}{t}\right)^tg\left(\frac{z_2-z_1}{1-t}\right)^{1-t}=f\otimes g(z_2,t)
\]
and, taking $x_0=(0,0,1/2)\in H$, then
\[
\int_{x_0+H^{\bot}}F(z_1,0,1/2)dz_1=\int_{\R^n}f(2z_1)^{\frac{1}{2}}g(-2z_1)^{\frac{1}{2}}dz_1=\frac{1}{2^n}\int_{\R^n}\sqrt{f(x)g(-x)}dx.
\]
Using Theorem \ref{th:splittRS} we obtain the desired inequality.

The case of equality is studied in Subsection \ref{subsec:equality in rsfunc(n+1)}.
\end{proof}

Next corollary shows that Theorem 2.1 in \cite{AGJV} is a direct consequence of Theorem \ref{th:splittRS}.

\begin{cor}\label{RSrecovered1}
Let $f,g\in\mathcal{F}(\R^n)$. Then
\begin{equation}\label{eq:RSfuncAGJV}
\|f*g\|_{\infty}\int_{\R^n}f\star g(z)dz\leq{2n\choose n}\|f\|_{\infty}\|g\|_{\infty}
\int_{\R^n}f(x)dx\int_{\R^n}g(y)dy.
\end{equation}
Equality holds if and only if $\frac{f(x)}{\|f\|_{\infty}}=\chi_K(x)=\frac{g(-x)}{\|g\|_{\infty}}$ and $K$ is a full-dimensional simplex.
\end{cor}

\begin{proof}
It is clear that the function
\[
F:\R^n\times\R^n\to[0,\infty),\quad F(z_1,z_2):=f(z_2)g(z_1-z_2),
\]
is a log-concave function. On the one hand
\[
\int_{\R^{2n}}F(z)dz=\int_{\R^n}f(z_2)\left(\int_{\R^n}g(z_1-z_2)dz_1\right)dz_2=\int_{\R^n}f(x)dx\int_{\R^n}g(y)dy.
\]
On the other hand, letting $H:=\lin\{e_1,\dots,e_n\}\in\L^{2n}_n$, then
\[
P_HF(z_1,0)=\sup_{z_2\in H^\bot}F(z_1,z_2)=\sup_{z_2\in H^\bot}f(z_2)g(z_1-z_2)=f\star g(z_1)
\]
and
\[
\max_{z_0\in H}\int_{H^{\bot}}F(z_0,z)dz=
\max_{z_0\in H}\int_{H^{\bot}}f(z)g(z_0-z)dz=\|f*g\|_{\infty}.
\]
Finally, $\|F\|_{\infty}=\|f\|_{\infty}\|g\|_{\infty}$. Theorem \ref{th:splittRS} applied onto a suitable translation of
$F$ so that the latter maximum is attained a t 0, and $H$ implies that
\begin{equation*}
\begin{split}
&\|f\|_{\infty}\|g\|_{\infty}\int_{\R^n}f(x)dx\int_{\R^n}g(y)dy=\|F\|_{\infty}\int_{\R^{2n}}F(z)dz\\
&\geq{2n\choose n}^{-1}\int_{H}P_HF(x)dx\max_{x_0\in
H}\int_{x_0+H^{\bot}}F(y)dy\\
&={2n\choose n}^{-1}\int_{\R^n}f\star g(x)dx\|f*g\|_{\infty},
\end{split}
\end{equation*}
proving the assertion.

Assume that there is equality in \eqref{eq:RSfuncAGJV}. Then $\frac{F}{\Vert F\Vert_\infty}=\chi_C$, for some $C\in\K^{2n}$. Consequently,
$$
F(z_1,z_2)=f(z_2)g(z_1-z_2)=\Vert f\Vert_\infty \Vert g\Vert_\infty \quad\text{for every }(z_1,z_2)\in C
$$
and $0$ elsewhere. Hence, $\frac{f}{\Vert f\Vert_\infty}=\chi_K$, $\frac{g}{\Vert g\Vert_\infty}=\chi_L$, for some $K,L\in\K^n$. Then
equality in \eqref{eq:RSfuncAGJV} becomes
\[
\max_{x_0}\vol(K\cap(x_0-L))\,\vol(K+L)={2n\choose n}\vol(K)\,\vol(L),
\]
which by the equality case of \eqref{eq:RogersShephard} holds if and only if $K=-L$ is an $n$-dimensional simplex.
\end{proof}

In \cite{AGJV}, the equality cases in \eqref{eq:RoShpolar} were obtained  as a consequence of the characterization of the equality cases in the Theorem 2.4 there (stated as the following corollary). Here we show that Theorem 2.4 in \cite{AGJV} is a consequence of Theorem \ref{SplittingColesanti} and that the characterization of the equality cases can be deduced from the ones in the geometric case \eqref{eq:RoShpolar}.

\begin{cor}
Let $f,g\in\mathcal{F}(\R^n)$ such that $f(0)=g(0)=\Vert f\Vert_\infty=\Vert g\Vert_\infty=1$. Then
\begin{equation}\label{eq:Thm2.4AGJV}
\int_{\R^n}\sqrt{f(x)g(-x)}dx\int_{\R^n}\sqrt{f\star g(2x)}dx\leq
2^n\int_{\R^n}f(x)dx\int_{\R^n}g(x)dx.
\end{equation}
Equality holds if and only if there exist simplices $L_1,L_2\in\K^n$,
having a common vertex at the origin, and their $n$ facets containing $0$ are contained
in the same $n$ hyperplanes, such that
$f(x)=\exp(-\Vert x\Vert_{L_1})$ and $g(x)=\exp(-\Vert x\Vert_{-L_2})$.
\end{cor}

\begin{proof}
Let $F(x,y):=f(x)g(y-x)$, $x,y\in\R^n$, and let $H:=\lin\{e_{n+1},\dots,e_{2n}\}\in\L^{2n}_n$. Theorem \ref{SplittingColesanti}
implies that
\[
\int_{\R^{2n}}F(x,y)dxdy\geq\frac{1}{2^{2n}}\int_{\R^n}P_HF(0,y)^{\frac{1}{2}}dx\int_{\R^n}F(x,0)^{\frac{1}{2}}dy.
\]
On the one hand
\[
\int_{\R^{2n}}F(x,y)dxdy=\int_{\R^n}f(x)dx\int_{\R^n}g(y)dy.
\]
Second,
\[
P_HF(0,y)^{\frac{1}{2}}=\sqrt{\sup_{x}f(x)g(y-x)}=\sqrt{f\star g(y)},
\]
and third,
\[
F(x,0)=\sqrt{f(x)g(-x)}.
\]
Since
\[
\frac{1}{2^n}\int_{\R^n}\sqrt{f\star g(y)}dy=\int_{\R^n}\sqrt{f\star g(2y)}dy,
\]
we conclude that
\[
\int_{\R^n}f(x)dx\int_{\R^n}g(y)dy\geq\frac{1}{2^n}\int_{\R^n}\sqrt{f\star g(2y)}dy
\int_{\R^n}\sqrt{f(x)g(-x)}dx.
\]

Equality in \eqref{eq:Thm2.4AGJV} implies that there is equality in Theorem \ref{SplittingColesanti}
for $F$, $H$, and $\lambda=1/2$. In particular, there exist $K_1,K_2\in\mathcal K^n$, $K_1\subset H^\bot$, $K_2\subset H$, such that
\[
f(x)g(-x)=\exp(-\Vert x\Vert_{K_1})\quad\text{and}\quad f\star g(y)=\exp(-\Vert y\Vert_{K_2}).
\]
If $f=e^{-u}$ and $g=e^{-v}$, for some $u,v$ convex functions, the first condition above rewrites as
\[
u(x)+v(-x)=\Vert x\Vert_{K_1} \quad\text{for every }x\in\R^n.
\]
Then for any $x\in\R^n$, $t\geq 0$, and $\lambda\in[0,1]$ we have that
\[
\begin{split}
((1-\lambda)+\lambda t)\Vert x\Vert_{K_1} & =  \Vert (1-\lambda)x+\lambda(tx)\Vert_{K_1}\\
& = u((1-\lambda)x+\lambda(tx))+v((1-\lambda)(-x)+\lambda(-tx)) \\
& \leq (1-\lambda)u(x)+\lambda u(tx)+(1-\lambda)v(-x)+\lambda v(-tx) \\
& =(1-\lambda)(u(x)+v(-x))+\lambda (u(tx)+v(-tx))\\
& = (1-\lambda)\Vert x\Vert_{K_1}+\lambda\Vert tx\Vert_{K_1}\\
& =((1-\lambda)+\lambda t)\Vert x\Vert_{K_1},
\end{split}
\]
which then implies equality in all inequalities above, thus
\[
u(tx)=t\,u(x)\quad\text{and}\quad v(tx)=t\,v(x)\quad\text{for every }x\in\R^n,\,t\geq 0.
\]
Denoting by $L_1:=\{x\in\R^n:u(x)\leq 1\}$ and $L_2:=\{x\in\R^n:v(x)\leq 1\}$, it is straightforward to show that the
equations above imply that
$u(x)=\Vert x\Vert_{L_1}$ and $v(x)=\Vert x\Vert_{L_2}$. Therefore
\[
\begin{split}
& n!\vol_n(\conv(\{L_1,L_2\})n!\vol_n\left(\left(\frac{L_1^\circ-L_2^\circ}{2}\right)^\circ\right) \\
& =\int_{\R^n}\sqrt{f\star g(2y)}dy\int_{\R^n}\sqrt{f(x)g(-x)}dx\\
& =2^n\int_{\R^n}f\int_{\R^n}g\\
& =2^nn!\vol_n(L_1)n!\vol_n(L_2),
\end{split}
\]
thus implying by the equality case of \eqref{eq:RoShpolar} that $L_1$ and $L_2$ are simplices with a
common vertex at the origin and such that the $n$-facets of $L_1$ and $-L_2$ touching $0$ are contained in the same hyperplanes.
\end{proof}

\section{Equality cases}\label{sec:Equality cases}

\subsection{Equality cases of Section \ref{sec:Estimates by marginals}}

\subsubsection{Equality cases in Lemma \ref{Ortholemma}}\label{subsect:equality in lemma Ortholemma}
For the equality case, if one of the bodies has empty interior then clearly both sides are $0$. Assume both are full-dimensional bodies, and there is equality for some $(x,y)\not\in L\times K$. If $x\in L$ and $y\not\in K$, by \eqref{eq:volumesplit} then
\begin{eqnarray*} \vol_{i+m}(\conv\{K\times\{x\} ,\{y\} \times L\}) &=&  \vol_{i+m}(\conv\{\conv(y,K)\times\{x\} ,\{y\} \times L\})\\&\geq& \binom{i+m}{i}^{-1}\vol_i(\conv(y,K))\,\vol_m(L)\\
&>& \binom{i+m}{i}^{-1}\vol_i(K)\,\vol_m(L)
\end{eqnarray*}
contradicting the equality (as $y\not\in K$) and similarly if $x\in K$ and $y\not\in L$. Finally, if
$x\not\in L$
and $y\not\in K$
we notice that by convexity there will be equality also for
$(x_\lambda, y_\lambda):= (1-\lambda)(x_0, y_0) + \lambda(x,y)$
for any $(x_0, y_0)\in L\times K$ and $\lambda \in (0,1)$. We can thus choose,  $x_0$ in the relative interior of $L$ and $y_0$ on the boundary of $K$ so that for some $\lambda$ we have $x_\lambda \in L $ and $y_\lambda\not\in K$, getting again a
contradiction.


\subsubsection{Equality case of Lemma  \ref{orthoRS}}\label{subsub:orthoRS}
Here it is the equality case in the lemma about the volume of the convex hull of two functions in orthogonal subspaces, namely Lemma \ref{orthoRS}.

Let us now assume that we have  equality in \eqref{eq:ProjSect}. On
the one hand, if we have equality in (c) then
\[
\int_{\R^i}\int_{\R^m}\min\left\{\frac{f(x)}{\Vert F\tilde{\star} G\Vert_\infty},\frac{g(y)}{\Vert F\tilde{\star} G\Vert_\infty}\right\}dxdy
=
\int_{\R^i}\frac{f(x)}{\Vert F\tilde{\star} G\Vert_\infty}dx\int_{\R^m}\frac{g(y)}{\Vert F\tilde{\star} G\Vert_\infty}dy,
\]
which by Lemma \ref{lem:MinLowerBound} implies that one of the functions in the minimum is a characteristic function. We can assume, without loss of generality, that
$\frac{g}{\Vert F\tilde{\star} G\Vert_\infty}$ is a characteristic function $\chi_L$ for some $L\in\K^m$ and
$\max\{\Vert f \Vert_\infty,\Vert g\Vert_\infty\}=\Vert g\Vert_\infty$. We will now show that  $f$ is a multiple of a characteristic too.

Equality on (b) implies, by Lemma \ref{Ortholemma}, that $0\in\{x\in\R^i:f(x)\geq t\Vert g\Vert_\infty\}\times L$, for every $t\in[0,\frac{\Vert f\Vert_\infty}{\Vert g\Vert_\infty}]$, hence implying that $f(0)=\|f\|_{\infty}$.

Besides, equality in (a) implies that $\Vert f\Vert_\infty= \Vert g\Vert_{\infty}$. Indeed, assume
that $\Vert f\Vert_\infty<\Vert g\Vert_\infty$. Then for every $t\in(\frac{\Vert f\Vert_{\infty}}{\Vert g\Vert_{\infty}},1)$ we have that
\[
\conv(\{(x,0)\in\R^i\times\{0\}:f(x)\geq t\Vert g\Vert_\infty\})\cup (\{0\}\times L)
\]
is an empty set and has volume $0$. On the contrary,
\[
\left\{(z_1,z_2):\sup_{0<\theta<1}f\left(\frac{z_1}{\theta}\right)^{\theta}
g\left(\frac{z_2}{1-\theta}\right)^{1-\theta}\geq t\Vert F\tilde{\star} G\Vert_\infty\right\}=\{(z_1,z_2):F\tilde{\star} G(z_1,z_2) \geq t\Vert F\tilde{\star} G\Vert_\infty\}
\]
has positive volume, which contradicts the equality in (a).

Let us now prove that $\frac{f}{\Vert f\Vert_\infty}$ is a characteristic function. Assume that there exists
$x_0\in\inter( \supp f)$ such that $0<\frac{f(x_0)}{\Vert f\Vert_{\infty}}=a<1$. Then since $f$ is continuous in $\inter(\supp f)$ for every $\varepsilon\in(0,\varepsilon_0)$ with $\varepsilon_0<\min\{1-a,a\}$ there exists a Euclidean ball $U_{\varepsilon}$ centered at $x_0$ such that
$a-\varepsilon < \frac{f(x)}{\Vert f\Vert_\infty}<a+\varepsilon<1$ for all $x\in U_{\varepsilon}$. Consequently,
\begin{itemize}
	\item $U_\varepsilon\cap \{x\in\R^i:f(x)\geq (a+\varepsilon)\Vert f\Vert_\infty\}=\emptyset$ and
	\item $U_\varepsilon\subseteq \{x\in\R^i:f(x)\geq (a-\varepsilon)\Vert f\Vert_\infty\}$.
\end{itemize}
Since $\lim_{\theta\rightarrow 0}(a-\varepsilon)^\theta=1$, there exists $\theta_0(\varepsilon)\in(0,1)$ such that
$a+\varepsilon<(a-\varepsilon)^\theta<1$ for every $\theta\in(0,\theta_0(\varepsilon))$. Consequently, for every
$x\in U_\varepsilon$ and $\theta\in(0,\theta_0(\varepsilon))$ then
$$
\left(\frac{f(x)}{\Vert f\Vert_\infty}\right)^\theta > (a-\varepsilon)^\theta > a+\varepsilon
$$
and then
$$
\left\{(z_1,z_2):\sup_{0<\theta<1}\left(\frac{f\left(\frac{z_1}{\theta}\right)}{\Vert f\Vert_\infty}\right)^\theta\chi_L^{1-\theta}\left(\frac{z_2}{1-\theta}\right)\geq a+\varepsilon\right\}
$$
contains
$$
\conv(((U_\varepsilon\cup\{x\in\R^i:f(x)\geq (a+\varepsilon)\Vert f\Vert_\infty\})\times\{0\})\cup (\{0\}\times L)).
$$
Now, since $U_\varepsilon\cap \{x\in\R^i:f(x)\geq (a+\varepsilon)\Vert f\Vert_\infty\}=\emptyset$ there exists an affine hyperplane
$H\subseteq\R^i$ separating both sets. Besides, since $\{0\}^i\in \{x\in\R^i:f(x)\geq (a+\varepsilon)\Vert f\Vert_\infty\}$ then
the volume of
\[
\conv(((U_\varepsilon\cup\{x\in\R^i:\frac{f(x)}{\Vert f\Vert_\infty}\geq a+\varepsilon\})\times\{0\})\cup (\{0\}\times L))
\]
is strictly larger than the volume of
\[
\conv((\{x\in\R^i:f(x)\geq (a+\varepsilon)\Vert f\Vert_\infty\}\times\{0\})\cup (\{0\}\times L)).
\]
Therefore, the volume of
\[
\left\{(z_1,z_2):\sup_{0<\theta<1}\left(\frac{f\left(\frac{z_1}{\theta}\right)}{\Vert f\Vert_\infty}\right)^\theta\chi_L^{1-\theta}\left(\frac{z_2}{1-\theta}\right)\geq a+\varepsilon\right\}
\]
is strictly larger than the volume of
\[
\conv((\{x\in\R^i:f(x)\geq (a+\varepsilon)\Vert f\Vert_\infty\}\times\{0\}^m)\cup (\{0\}^i\times L))
\]
for every $\varepsilon\in(0,\varepsilon_0)$, which contradicts the equality in (a). Thus
$\frac{f}{\Vert f\Vert_\infty}$ takes values only 0 or 1, and hence it is the characteristic function of a convex set.

\subsubsection{Equality case of Theorem \ref{th:splittRS}}\label{subsub:splittRS}

Equality holds in \eqref{eq:splittingTheorem} if and only if there is equality for every inequality in the proof.
Since we have that
\[
\Vert f\Vert_{\infty}\int_{\R^n}F\tilde{\star}G(z)dz = {n\choose i}^{-1} \int_{H}P_Hf(x)dx \int_{H^\bot }f(y)dy,
\]
equality in Lemma \ref{orthoRS} implies that $\Vert P_Hf|_{H^{\bot}}\Vert_\infty=\Vert f|_{H^{\bot}}\Vert_\infty=\Vert f\Vert_\infty$,
and there exist $L_1\in\K^i$ and $L_2\in\K^{n-i}$ such that $0\in L_1$, $0\in L_2$, and
$$
\frac{P_Hf}{\Vert f\Vert_\infty}=\frac{S_Hf|_H}{\Vert f\Vert_\infty}=\chi_{L_1}\quad\text{and}\quad
\frac{S_{H}f|_{H^\bot}}{\Vert f\Vert_\infty}=\chi_{L_2}.
$$
Furthermore, notice that $L_2$ has to be a Euclidean ball. Since we also have that
\[
\int_{\R^n}f(z)dz=\int_{\R^n}S_Hf(z)dz=\int_{\R^n}F\tilde{\star}G(z)dz,
\]
we then have that
$$
S_Hf=F\tilde{\star}G=\Vert f\Vert_\infty\chi_C,
$$
with
$$
C=\conv(\{L_1\times\{0\},\{0\}\times L_2\}).
$$
Consequently there exists $K\in\K^n$ such that $\frac{f}{\Vert f\Vert_\infty}=\chi_K$,
where $S_H(K)=C$. Hence equality in \eqref{eq:splittingTheorem} reads as
$$
\vol(K)={n\choose i}^{-1}\vol(P_HK)\max_{x_0\in H}\vol(K\cap H^\bot).
$$
By the characterization of equality in \eqref{eq:volumesplit} we obtain that for every $v\in H$ the intersection $K\cap (H^{\bot}+\R^+v)$ is the convex hull of $K\cap H^{\bot}$ and one point.

\subsubsection{Equality case of Lemma \ref{th:SplittColesLambda}}\label{subsub:lemasplittcoleslambda}

Let $u:\R^n\rightarrow[0,\infty]$ be a convex function such that
$\frac{f}{\Vert f\Vert_\infty}=e^{-u}$. We can assume, without loss of generality, that $\Vert f\Vert_\infty=1$ and then $u(0)=0=\min u(z)$. $f$ attains equality in \eqref{splittingColLambda} if and only if
\[
u\left((1-\lambda)\left(\frac{z_1}{1-\lambda},0\right)+\lambda\left(0,\frac{z_2}{\lambda}\right)\right)=(1-\lambda) u\left(\frac{z_1}{1-\lambda},0\right)+\lambda u\left(0,\frac{z_2}{\lambda}\right)
\]
for every $(z_1,z_2)\in\R^n$. Fixing $z_2=0$, this means that for every $z_1\in H$
\[
u\left((1-\lambda)\frac{z_1}{1-\lambda},0\right)=(1-\lambda) u\left(\frac{z_1}{1-\lambda},0\right).
\]
Since this is true for every $z_1\in H$ and $u(0)=0$, then for every $t\geq0$ and every $z_1\in H$
\[
u(tz_1,0)=t\,u(z_1,0).
\]
Letting $K:=\{z_1\in H:u(z_1,0)\leq 1\}$, notice that $0\in K$ and
we conclude that $u(z_1,0)=\Vert z_1\Vert_K$ for every $z_1\in H$.
By analogous arguments, we have for
$L:=\{z_2\in H^\bot:u(0,z_2)\leq 1\}$ that $u(0,z_2)=\Vert z_2\Vert_L$ for every $z_2\in H^\bot$.
Finally, we also obtain that for every $z=(z_1,z_2)\in\R^n$, since $z=(1-\lambda)\left(\frac{z_1}{1-\lambda},0\right)+\lambda\left(0,\frac{z_2}{\lambda}\right)$
\[
\begin{split}
\frac{f(z)}{\Vert f\Vert_\infty}=\exp(-u(z_1,z_2)) & =\exp\left(-(1-\lambda)u\left(\frac{z_1}{1-\lambda},0\right)-\lambda u\left(0,\frac{z_2}{\lambda}\right)\right)\\
& = \exp\left(-(1-\lambda)\left\Vert \frac{z_1}{1-\lambda}\right\Vert_K-\lambda\left\Vert \frac{z_2}{\lambda}\right\Vert_L\right)) \\
& =\exp(-\Vert z_1\Vert_K-\Vert z_2\Vert_L),
\end{split}
\]
hence concluding the proof.

\subsubsection{Equality case of Theorem \ref{SplittingColesanti}}\label{subsub:thsplittcoleslambda}

We can assume, without loss of generality, that $\Vert f\Vert_\infty=f(0)=1$. Let us observe that if we have equality in Theorem \ref{SplittingColesanti}, then
we have equality in Lemma \ref{th:SplittColesLambda} for $\tilde{f}=S_Hf$.
This means that there exist $K_1\subset H$ and $K_2\subset H^\bot$ with $0\in K_1\cap K_2$, such that for every $(x,y)\in H\times H^\perp$, $\tilde{f}(x,y)=\exp(-\Vert x\Vert_{K_1}-\Vert y\Vert_{K_2})=e^{-\Vert (x,y)\Vert_L}$, where
$$
L=\conv(K_1, K_2).
$$

If $u:H\times H^\perp\to[0,\infty]$ is such that $f=e^{-u}$ we have that $S_H(\epi(u))=\{(x,y,t):\Vert(x,y)\Vert_L\leq t\}$. Consequently, since for every $t_0\in[0,\infty)$ we have $\vol_n(\epi(u)\cap\{t=t_0\})=\vol_n(S_H(\epi(u)\cap\{t=t_0\}))=t_0^n\vol_n(L)$, by the equality cases in Brunn-Minkowski inequality, there exists a convex body $K\subseteq\R^n$ with $0\in K$ such that
$\epi(u)\cap\{t=t_0\}=t_0K$ and $L=S_H(K)$. Thus, $f(x,y)=e^{-\Vert(x,y)\Vert_K}$ for some $K$ with $0\in K$ and
\[
\conv(K_1, K_2)=S_{H}(K).
\]
Consequently, the equality case in Theorem \ref{SplittingColesanti} becomes equality in \eqref{eq:volumesplit} and then for every $v\in H$ the intersection $K\cap (H^{\bot}+\R^+v)$ is the convex hull of $K\cap H^{\bot}$ and one point.



\subsubsection{Equality case in Lemma \ref{lem:MinUpperBound}}\label{subsub:lemaminupperbound}

Assume that there is equality in Lemma \ref{lem:MinUpperBound}.
We can assume without loss of generality that $\Vert f\Vert_\infty=f(0)$ and $\Vert g\Vert_\infty=g(0)$ and so $0\in K_t\cap L_t$ for every $t\in (0,1]$. For every $t\in(0,1]$ we have that
$$
\vol_{n+m}(C_t) ={n+m\choose n}^{-1}\vol_n(K_t)\vol_{m}(L_t).
$$
By the equality cases in Rogers-Shephard inequality \eqref{eq:volumesplit} this implies that for every $t\in(0,1]$ and every $(x,0)\in P_HC_t=K_t\times\{0\}$,
\begin{equation}\label{eq:KsAndLs}
L_{t\frac{\Vert f\Vert_\infty}{f(x)}}=y_x+(1-\Vert x\Vert_{K_t})L_t.
\end{equation}
for some $y_x\in\R^m$. Since this is true for every $t\in(0,1]$ and  every $x\in K_t$ we deduce that all the convex bodies $L_t$ are homothetic and there exists a function $g_1(t)$ and a convex body $L$ with $0\in L$ such that $L_t=y_t+g_1(t)L$. Notice also that, taking for any $t\in(0,1]$ some $x$ with $\Vert x\Vert_{K_t}=1$  we deduce that $L_{t\frac{\Vert f\Vert_\infty}{f(x)}}=\{y_x\}$ and so $y_x=0$ for every $x\in \inter(\supp f)$  and then $y_t=0$ for every $t\in(0,1]$. Furthermore, choosing some $t$ and $x$ such that $f(x)=t\Vert f\Vert_\infty$ we deduce that $L_1=\{0\}^m$.
Besides, since $g$ is a log-concave function, we have that for any $v_1,v_2\in[0,\infty)$ and any $\lambda\in[0,1]$
$$
L_{e^{-((1-\lambda)v_1+\lambda v_2)}}\supseteq (1-\lambda)L_{e^{-v_1}}+\lambda L_{e^{-v_2}}
$$
and then the function $G(v):=g_1(e^{-v})$ is concave and verifies that $G(0)=0$.

Now, for any $t\in(0,1]$, take $s\geq t$ and $\theta\in S^{n-1}$. Notice that, from \eqref{eq:KsAndLs} one can deduce that for any ray starting at 0, its intersection with $\supp f$ is either the whole ray or just $\{0\}$. Otherwise, fix some $x$ in the ray and take $t\to0$, which leads to a contradiction. Consequently $\Vert\theta\Vert_{K_s}$ is finite if and only if $\Vert\theta\Vert_{K_t}$ is finite and, in such case $f(\lambda\theta)\to0$ as $\lambda\to\infty$ and $\Vert\theta\Vert_{K_s}$ is strictly greater that $\Vert\theta\Vert_{K_t}$. Taking $x=\frac{\theta}{\Vert\theta\Vert_{K_s}}$ we have
$$
L_{\frac{t}{s}}=\left(1-\frac{\Vert \theta\Vert_{K_t}}{\Vert\theta\Vert_{K_s}}\right)L_t
$$
and then for every $\theta\in S^{n-1}$ and every $0<t\leq s\leq1$
$$
\Vert\theta\Vert_{K_t}=\left(1-\frac{g_1\left(\frac{t}{s}\right)}{g_1(t)}\right)\Vert\theta\Vert_{K_s}.
$$
Thus, there exists a function $f_1(t)$ and a convex body $K$ with $0\in K$ such that $K_t=f_1(t)K$ and, since $f$ is log-concave, the function $F(u)= f_1(e^{-u})$ is concave. Then, for any $0<t\leq s\leq1$
$$
\frac{f_1(s)}{f_1(t)}=\left(1-\frac{g_1\left(\frac{t}{s}\right)}{g_1(t)}\right),
$$
or equivalently, taking $t=e^{-v}$ and $s=e^{-u}$, for any $0\leq u<v<\infty$
$$
\frac{F(u)}{F(v)}+\frac{G(v-u)}{G(v)}=1.
$$
On the one hand, we deduce that
$$
\frac{F(v)}{G(v)}\frac{G(v-u)}{v-u}=\frac{F(v)-F(u)}{v-u},
$$
and, since $F$ is concave, taking $v-u$ constant we deduce that $\frac{F(v)}{G(v)}$ is non-increasing. On the other hand
$$
\frac{G(v)}{F(v)}\frac{F(u)}{u}=F(v)\frac{G(v)-G(v-u)}{u}
$$
and, since $G$ is concave, taking $u$ constant we deduce that $\frac{G(v)}{F(v)}$ is non-increasing and, since both $F$ and $G$ are positive functions, $\frac{F(v)}{G(v)}$ is non-decreasing and thus it is constant. Then $G=CF$, for some $C>0$, $F(0)=0$, and for every $0\leq u<v<\infty$
$$
F(u)+ F(v-u)=F(v),
$$
so $F(u)=au$ and $G(v)=bv$ for some positive constants $a,b$ and then $f_1(t)=-a\log t$ and $g_1(t)=-b\log t$. Consequently, for every $t\in (0,1]$ $K_t=(-\log t)(aK)$ and $L_t=(-\log t)(bL)$, which happens if and only if $\frac{f(x)}{\Vert f\Vert_\infty}=\exp(-\Vert x\Vert_{aK})$ and $\frac{g(y)}{\Vert f\Vert_\infty}=\exp(-\Vert y\Vert_{bL})$.

\subsubsection{Equality case in Theorem \ref{th:upperboundexponent}}\label{subsub:thmupperboundexponent}

In order to have equality, by the equality cases of Lemma \ref{lem:MinUpperBound}, it has to be $\frac{P_Hf(x)}{\Vert f\Vert_\infty}=\exp(-\Vert x-x_0\Vert_K)$, $\frac{P_{H^\bot}(y)}{\Vert f\Vert_\infty}=\exp(-\Vert y-y_0\Vert_L)$ for some $x_0\in H, y_0\in H^\perp$ and $K,L$ convex bodies in $H$ and $H^\perp$ respectively with the origin in their interiors. Hence
\begin{eqnarray*}
\frac{f(x,y)}{\Vert f\Vert_\infty}&=&\min\left\{\frac{P_Hf(x)}{\Vert f\Vert_\infty}, \frac{P_{H^\bot}(y)}{\Vert f\Vert_\infty}\right\}= \min\left\{e^{-\Vert x-x_0\Vert_K}, e^{-\Vert y-y_0\Vert_L}\right\}\cr
&=&e^{-\max\{\Vert x-x_0\Vert_K,\Vert y- y_0\Vert_L\}}=e^{-\Vert z-z_0\Vert_{K\times L}},
\end{eqnarray*}
where $z_0=(x_0,y_0)$.

\subsection{Equality cases of Subsection \ref{subsec:LWreverse}}\label{subsec:Equality in RSwithprojandintersection}

\subsubsection{Equality case of Lemma \ref{RSwithIntersection}}

Let us assume that there is equality for $K$  in \eqref{eq:revLWandRS}. First, since
\[
\|f\|_{\infty}\int_{(E\cap H)^{\bot}}f(z)dz = {n-k\choose n-i}^{-1}\int_{H^{\bot}}P_{H^{\bot}}f(x)dx\int_{E^{\bot}}f|_{E^{\bot}}(y)dy,
\]
by Theorem \ref{th:splittRS} we have that $f=\|f\|_{\infty}\chi_{P_{(E\cap H)^{\bot}}K}$,
and $L:=P_{(E\cap H)^{\bot}}K$ verifies that for every $v\in H^\bot$ then
$L\cap(E^\bot+\R^+v))$ is the convex hull of $L\cap E^\bot$ and one more point. 
$\max_{x_0\in H^\bot}\vol(L\cap(x_0+E^\bot))=\vol(L\cap E^\bot)$.

Besides, since $f$ is a constant function on $L$, by the equality case of Brunn-Minkowski inequality,
$K\cap(x+E\cap H)$ is a translate of the same convex body $K_3\subseteq E\cap H$, for every $x\in L$.

Since we also have that
\begin{eqnarray*}
\vol_k(K_3)\vol_{n-j}(P_{H^\bot}L)&=&\int_{H^{\bot}}\max_{y\in E^\bot}\vol_k(K\cap(x+y+E\cap H))dx = \int_{H^{\bot}}\vol_k(K\cap(x+E\cap H))dx\cr
&=&\vol_k(K_3)\vol_{n-j}(L\cap H^\bot)
\end{eqnarray*}
then $P_{H^\bot}K=P_{H^\bot}L=L\cap H^\bot$. Let us define $K_1:=L\cap E^\bot$ and $K_2:=P_{H^\bot}K$. Since $K_1$ and $K_2$
are contained in $L$, the convex hull is contained in $L$ as well, and since for every $v\in H^\bot$
$L\cap(E^\bot+\R^+v))$ is the convex hull of $L\cap K_1$ and one more point, then
$L=\conv(\{K_1,K_2\})$.

\subsubsection{Equality case of Theorem \ref{th:localLoomisWhitneyReverse}}

In order to have equality in Theorem \ref{th:localLoomisWhitneyReverse} for some $K$,
then $\widetilde{K}=S_EK$ must attain equality in Lemma \ref{RSwithIntersection}.
Let us fix $x\in P_{H^\bot}\widetilde{K}=P_{H^\bot}K$. Then we have that for every $y\in E^\bot$ such that $x+y\in P_{(E\cap H)^\bot}\widetilde{K}$,
$\widetilde{K}\cap(x+y+E\cap H)$ is a translate of the same convex body $K_1\subset E\cap H$.
Since $\widetilde{K}$ is symmetric with respect to $E$, we actually have that for every $z\in E\cap H$ such that $x+z\in P_E\widetilde{K}=P_EK$.
$\widetilde{K}\cap(x+z+E^\bot)$ is a translate of the same convex body. Therefore,
$\vol_{n-i}(K\cap(x+z+E^\bot))=\vol_{n-i}(\widetilde{K}\cap(x+z+E^\bot))$ is constant in its support and we also have that for a fixed $x\in P_{H^\bot}K$,
$K\cap(x+z+E^\bot)$ is a translate of the same body $K_{2,x}$ for every $z\in E\cap H$ such that $x+z\in P_EK$.

Let $L:=P_{(E\cap H)^\bot}\widetilde{K}$. Equality in Lemma \ref{RSwithIntersection} also ensures that $L=\textrm{conv}(\{L\cap E^\bot, P_{H^\bot}L\})$ and then for every $x\in P_{H^\bot}K=P_{H^\bot} L$
$$
\vol_{n-i}(L\cap (x+E^\bot))=(1-\Vert x\Vert_{P_{H^\bot}L})^{n-i}\vol_{n-i}(L\cap E^\bot).
$$
Since $\vol_{n-i}(L\cap (x+E^\bot))=\vol_{n-i}(K_{2,x})$ and $K_3:=P_{H^\bot}L=P_{H^\bot}K$ we have that
$$
\vol_{n-i}(K_{2,x})=(1-\Vert x\Vert_{K_3})^{n-i}\vol_{n-i}(K_{2,0}).
$$
Therefore, since $\vol_{j}(K\cap(x+H))=\vol_j(\widetilde{K}\cap(x+H))=\vol_{n-i}(K_{2,x})\vol_k(K_1)$, we have that for any $x\in K_3$
$$
\vol_i(K\cap(x+H))=(1-\Vert x\Vert_{K_3})^{n-i}\vol_i(K\cap H).
$$

Notice that for every $x_0\in K$ with $\Vert x_0\Vert_{K_3}=1$ we have that $C:=\textrm{conv}\{{K\cap H,\{x_0\}}\}\subseteq K$ and that for any $x=\lambda P_{H^\bot}x_0$ with $0\leq\lambda\leq 1$
$$
\vol_i(C\cap(x+H))=(1-\Vert x\Vert_{K_3})^{n-i}\vol_i(C\cap H)=\vol_i(K\cap(x+H)).
$$

Therefore $C\cap(x+H)=K\cap(x+H)$ and then for every $v\in H^\perp$, $K\cap (H+\R^+v)$ is the convex hull of $K\cap H$ and one point, where $K\cap H$ verifies that $S_E(K\cap H)=K_{2,0}\times K_1$

\subsection{Equality case of Theorem \ref{th:localLoomisWhitney}}\label{subsec:equality in directLW}

Let us assume that there is equality in \eqref{dirLocLW}. Equality in \eqref{LocalLWeq2} implies that
$f_1$ and $f_2$ attain their maximum at the same point $x_0\in P_{E\cap H}K$. Let
\begin{itemize}
	\item $K_1:=P_EK\cap(x_0+H^\bot)-x_0\subseteq H^\bot$,
	\item $K_2:=P_HK\cap(x_0+E^\bot)-x_0\subseteq E^\bot$ and
	\item $K_3:=P_{E\cap H}K$.
\end{itemize}
Besides, for every $x^*\in\partial K_3$ and $\lambda\in[0,1]$ then $f_i((1-\lambda)x_0+\lambda x^*)=(1-\lambda)f_i(x_0)$, $i=1,2$.

By the equality cases in Brunn-Minkowski inequality, for every $x^*\in\partial K_3$ we have that
$P_EK\cap(((1-\lambda)x_0+\lambda x^*)+H^\bot)$ and $P_HK\cap(((1-\lambda)x_0+\lambda x^*)+E^\bot)$ are
translates of $(1-\lambda)K_1$ and $(1-\lambda)K_2$ respectively.

Equality in \eqref{LocalLWeq1} implies that for every $x\in P_{E\cap H}K$
$$
K\cap(x+(E\cap H)^\bot)=(P_EK\cap(x+H^\bot))+(P_HK\cap(x+E^\bot)).
$$
Consequently, for every $x^*\in\partial K_3$
\[
K\cap(((1-\lambda)x_0+\lambda x^*)+(E\cap H)^\bot)
\]
is a translate of $(1-\lambda)(K_1+K_2)$. Hence we can conclude that
for every $v\in E\cap H$ we have that $K\cap(x_0+(E\cap H)^\bot)+\R^+v$ is the convex hull of $x_0+(K_1+K_2)$ and a unique point.

\subsection{Equality in Theorem \ref{thm:RSFuncconvexhullbest}} \label{subsub:equalRSFuncConvHullBest}

Let us now suppose that we have equality in \eqref{eq:RSFuncconvexhullbest}.
This means, in particular, that there exists $C\in\mathcal K^{2n+1}$ such that
\[
\frac{f\left(\frac{z_1}{1-t}\right)^{1-t}g\left(\frac{z_2}{t}\right)^t}{\max\{\Vert f\Vert_\infty,\Vert g\Vert_\infty\}}=
\frac{\Psi(z_1,z_2,t)}{\Vert \Psi\Vert_\infty} =\chi_C(z_1,z_2,t).
\]
Hence there exist $K,L\in\mathcal K^n$ such that $\frac{f}{\Vert f\Vert_\infty}=\chi_K$
and $\frac{g}{\Vert g\Vert\infty}=\chi_L$. Moreover, since for every $t\in[0,1]$
\[
1=\frac{\Vert f\Vert_\infty^t\Vert g\Vert_\infty^{1-t}}{\max\{\Vert f\Vert_\infty,\Vert g\Vert_\infty\}}
\]
then $\Vert f\Vert_\infty=\Vert g\Vert_\infty$
Replacing $f$ and $g$ by these characteristic functions, we get that
\[
\int_0^1t^n(1-t)^n\int f(x)^{1-t}dx\int g(y)^tdy=\frac{\Vert f\Vert_\infty}{(n+1){2n+1\choose n}}\vol_n(K)\vol_n(L).
\]
Moreover, we have that $\chi_K\tilde{\star}(\chi_L)_-(x)=\chi_M(x)$,
where $M$ is given by
\[
M=\{z\in\R^n:z=(1-\lambda)x+\lambda y,\,x\in K,\,y\in -L,\,\lambda\in[0,1]\}=\conv(K\cup (-L)),
\]
and that
\[
\sup_{0\leq s\leq 1}\chi_K\left(\frac{x}{1-s}\right)^{1-s}\chi_L\left(\frac{x}{s}\right)^s=\chi_N(x),
\]
where
\[
N=\bigcup_{0\leq s\leq 1}(((1-s)K)\cap(s L))=(K^\circ+L^\circ)^\circ.
\]
Therefore equality \eqref{eq:RSFuncconvexhullbest} rewrites as
\[
\vol_n(\conv(K\cup (-L)))\vol_n((K^\circ+L^\circ)^\circ)=\frac{(2n+1){2n\choose n}}{(n+1){2n+1\choose n}}\vol_n(K)\vol_n(L)=\vol_n(K)\vol_n(L).
\]
which, by the equality case of \eqref{eq:RoShpolar}, holds if and only if $K$ and $L$ are simplices with a common vertex at the origin and such that the $n$ facets of $K$ and $-L$ containing the origin are contained in the same set of $n$ hyperplanes.

\subsection{Equality of Theorem \ref{thm:RSFunc(n+1)} }\label{subsec:equality in rsfunc(n+1)}

Let us assume that there is equality in \eqref{eq:RSFunc(n+1)}.
Then $F$ attains equality in Theorem \ref{th:splittRS}, hence there exists $C\in\mathcal{K}^{2n+1}$
such that $F=\Vert F\Vert_\infty\chi_C$. Since for every $(z_1,z_2,t)\in C$
\[
1=\frac{F(z_1,z_2,t)}{\Vert F\Vert_\infty}=\frac{f(\frac{z_1}{t})^tg(\frac{z_2}{1-t})^{1-t}}{\Vert F\Vert_\infty},
\]
then there exist $K,L\in\mathcal K^n$ such that $\frac{f}{\Vert f\Vert_\infty}=\chi_K$
and $\frac{g}{\Vert g\Vert\infty}=\chi_L$. Moreover, since for every $t\in[0,1]$
\[
1=\frac{\Vert f\Vert_\infty^t\Vert g\Vert_\infty^{1-t}}{\max\{\Vert f\Vert_\infty,\Vert g\Vert_\infty\}}
\]
then $\Vert f\Vert_\infty=\Vert g\Vert_\infty=A$. Now,
\[
\frac{f\otimes g(z_2,t)}{A}=\sup_{z_1\in\R^n}\chi_K\left(\frac{z_1}{t}\right)\chi_L\left(\frac{z_2-z_1}{1-t}\right)=1
\]
occurs if and only if there exists $z_1\in\R^n$ s.t.~$z_1\in tK$ and $z_2\in tK+(1-t)L$, which means that
$\frac{f\otimes g(z_2,t)}{A}=\chi_{tK+(1-t)L}(z_2)$, and thus
\begin{equation*}
\begin{split}
&\int_{\R^n\times[0,1]} \frac{f\otimes g(z_2,t)}{A}dz_2dt=\int_0^1\int_{\R^n}\chi_{tK+(1-t)L}(z_2)dz_2dt\\
&=\int_0^1\vol_n(tK+(1-t)L)dt=\vol_{n+1}(\conv\{K\times\{1\},L\times\{0\}\}).
\end{split}
\end{equation*}
Second, $\frac{\sqrt{f(x)g(-x)}}{A}=\chi_K(x)\chi_{-L}(x)=\chi_{K\cap(-L)}(x)$.
Thus
\[
\int_{\R^n} \frac{\sqrt{f(x)g(-x)}}{A}dx=\int_{\R^n} \chi_{K\cap(-L)}(x)dx=\vol_n(K\cap(-L)).
\]
Since we also get that
\begin{eqnarray*}
\int_0^1t^n(1-t)^n\int_{\R^n}\left(\frac{f(x)}{A}\right)^tdx\int_{\R^n}\left(\frac{g(y)}{A}\right)^{1-t}dydt&=&\int_0^1t^n(1-t)^ndt\,\vol(K)\,\vol(L)\cr
&=&\frac{{2n+1\choose n}^{-1}}{n+1}\vol_n(K)\,\vol_n(L),
\end{eqnarray*}
altogether shows that equality in \eqref{eq:RSFunc(n+1)} becomes an equality in \eqref{eq:RoShHyperplanes}, hence
concluding that $K=-L$ is an $n$-dimensional simplex.

\section{Appendix: Berwald's inequality}\label{sec:appendixBerwald}


As it was said above, this appendix is devoted to present a comprehensive self-contained proof of \cite[Satz 8]{Ber}, so far and to the best of our knowledge, not yet found in English. We try to keep the original ideas and notations as accurate as possible to the ones of Berwald.

Let $K\in\mathcal K^n$. For $\hat M>0$ and $x_0\in K$,
the roof function on $K$ with height $\hat M$ over $x_0\in K$ is a function $\hat f_{\hat M}(\cdot;x_0):K\to[0,+\infty)$
such that the graph of $\hat f=\hat f_{\hat M}(\cdot;x_0)$ in $\R^{n+1}$ is a hypercone with basis $K$ and height $\hat M$, such that the projection of the vertex is $x_0\in K$.

In other words,
$$
\{(x,t)\in K\times \R: 0\le t\le \hat f(x)\}
=
\text{conv}(K\times \{0\},\{(x_0,\hat M)\}).
$$
Also, for $0\le t\le \hat M$,
$$
\{x\in K: \hat f(x)\ge t\}=\tfrac{t}{\hat M}x_0+(1-\tfrac{t}{\hat M})K.
$$

\begin{thm}\label{berwald}
Let $K\in\mathcal K^n$, $f_1,\dots f_m:K\to[0,+\infty)$ concave, continuous, and non identically null functions, and $\alpha_1,\dots, \alpha_m>0$. Then
$$
\frac{1}{\vol_n(K)}\int_K \prod_{i=1}^m f_i(x)^{\alpha_i}\,dx
\le
\frac{{\alpha_1+n\choose n}\cdots {\alpha_m+n\choose n}}
{{\alpha_1+\cdots +\alpha_m+n\choose n}}
\prod_{i=1}^m \frac{1}{\vol_n(K)}\int_Kf_k(x)^{\alpha_k}\,dx.
$$
Equality holds for $m>1$ if and only if all the $f_i$'s are roof functions over the same point in $K$.
\end{thm}

We need to state several results translated from \cite{Ber} (written in \emph{old german})
before giving a proof of Theorem \ref{berwald}.

\begin{thm} \label{nvariable}
Let $0<\alfa_1<\alfa_2$, $K\in\mathcal K^n$, and $f:K\to[0,+\infty)$ concave, continuous, and not identically null. Then
\begin{equation}\label{decreasing2}
\left(\frac{{\alfa_2+n\choose n}}{\vol_n(K)}\int_K f(x)^{\alfa_2}\,dx\right)^{1/\alfa_2}
\le
\left(\frac{{\alfa_1+n\choose n}}{\vol_n(K)}\int_K f(x)^{\alfa_1}\,dx\right)^{1/\alfa_1}.
\end{equation}
Equality holds if and only if $f$ is a roof function over a point in $K$.
\end{thm}

\begin{proof}
Let $M$ be the maximum of $f$ on $K$. For $t\ge0$, let
$$
V_f(t)=\vol_n(K_t)
$$
where
$$
K_t=\{
x\in K: f(x)\ge t
\}.
$$
The function $V_f$  is continuous on $[0,M]$, non-negative, non-increasing, $V_f(0)=\vol(K)$ and  for $t>M$, $V_f(t)=0$.
The concavity of $f$ and Brunn-Minkowski inequality show that $V_f^{1/n}$ is a concave function on $[0,M]$.

For $\alfa>0$, let
$$
\Phi_\alfa(f)=\left(\frac{{\alfa+n\choose n}}{\vol(K)}\int_{B_f} dx\,d(t^\alfa)\right)^{1/\alfa}
$$
where
\begin{equation}
\label{setB}
B_f=\{(x,t)\in K\times \R: 0\le t\le f(x)\}
\end{equation}
is a convex set in $\R^{n+1}$.
We compute the $(n+1)$-dimensional integral  in two  ways:
\begin{equation}\label{nequality1}
\Phi_\alfa(f)^\alfa=\frac{{\alfa+n\choose n}}{\vol(K)}\int_K f(x)^\alfa\,dx
=
\frac{{\alfa+n\choose n}}{\vol(K)}\int_0^M V_f(t)\,d(t^\alfa).
\end{equation}
In particular, the roof function $\hat f=\hat f_{\hat M}(\cdot;x_0)$ with height $\hat M>0$ (arbitrary for the moment) over any point $x_0\in K$ has the same function:
\begin{equation}\label{vroof}
V_{\hat f} (t)=\left(1-\frac{t}{\hat M}\right)^n\vol(K).
\end{equation}
Integrating by parts,
\begin{equation}\label{varphi}
\begin{aligned}
\Phi_\alfa(\hat f)^\alfa=
\frac{{\alfa+n\choose n}}{\vol(K)}\int_K \hat f(x)^\alfa\,dx
&=
\frac{{\alfa+n\choose n}}{\vol(K)}\int_0^{\hat M} V_{\hat f}(t)\,d(t^\alfa)
\\
&=\frac{n{\alfa+n\choose n}}{\hat M^n}\int_0^{\hat M}t^\alfa(\hat M-t)^{n-1}\,dt=
\hat M^\alfa.
\end{aligned}
\end{equation}
Take $\hat M=\Phi_{\alfa_1}(f)>0$.
From $\Phi_{\alfa_1}(\hat f)=\Phi_{\alfa_1}(f)$, we get
\begin{equation}\label{Gzero}
\int_0^{\hat M} V_{\hat f}(t)\,d(t^{\alfa_1})=\int_0^{ M} V_f(t)\,d(t^{\alfa_1}).
\end{equation}

Suppose $f$ is not a roof function. Then $\hat M>M$. Indeed, assume $\hat M\le M$. The convexity of $B_f$ implies $B_{\hat f}\subset B_f$ (strict inclusion). Then we have $\Phi_{\alfa_1}(\hat f)<\Phi_{\alfa_1}(f)$, which is a contradiction.

Since $V_f^{1/n}$ is concave on $[0,M]$, and $V_f(0)^{1/n}=\vol(K)^{1/n}$, $V_f(t)^{1/n}=0$ for $t>M$, the functions
$$
V_{\hat f}(t)^{1/n}=(1-\frac{t}{\hat M})\vol(K)^{1/n}
$$
and $V_f^{1/n}$ switch in just one point $t_0\in [0,M]$, i.e.,
$$
V_f(t)- V_{\hat f}(t)
\begin{cases}
>0 &\text{if } 0<t<t_0\\<0 &\text{if } t_0<t<\hat M.
\end{cases}
$$

So, using (\ref{Gzero}), and extending both integrals to $[0,\hat M]$ (since $V(t)=0$ for $t>M$),
$$
0=
\int_0^{\hat M}( V_f(t)- V_{\hat f}(t))\,d(t^{\alfa_1})
=
\int_0^{t_0} ( V_f(t)- V_{\hat f}(t))\,d(t^{\alfa_1})-\int_{t_0}^{\hat M} ( V_{\hat f}(t)- V_f(t))\,d(t^{\alfa_1})
$$
so
\begin{equation}\label{equality}
\int_0^{t_0} ( V_f(t)- V_{\hat f}(t))\,d(t^{\alfa_1})
=
\int_{t_0}^{\hat M} ( V_{\hat f}(t)- V_f(t))\,d(t^{\alfa_1}).
\end{equation}
Now take $\alfa=\alfa_2$ in (\ref{nequality1}),
$$
\Phi_{\alfa_2}(f)^{\alfa_2}-\Phi_{\alfa_2}(\hat f)^{\alfa_2}
=\frac{{\alfa_2+n\choose n}}{\vol(K)}\left[
\int_0^M V_f(t)\,d(t^{\alfa_2})-\int_0^{\hat M}  V_{\hat f}(t)\,d(t^{\alfa_2})\right]
$$
and again extending the interval of integration to $[0,\hat M]$,
$$
=
\frac{{\alfa_2+n\choose n}}{\vol(K)}\int_0^{\hat M}(V_f(t)- V_{\hat f}(t))\,d(t^{\alfa_2})
$$
$$
=\frac{{\alfa_2+n\choose n}}{\vol(K)}
\left[\int_0^{t_0} (V_f(t)- V_{\hat f}(t))\,d(t^{\alfa_2})-\int_{t_0}^{\hat M} (V_{\hat f}(t)- V_f(t))\,d(t^{\alfa_2})\right]
$$
Using $d(t^{\alfa_2})=\frac{\alfa_2}{\alfa_1}t^{\alfa_2-\alfa_1}d(t^{\alfa_1})$ and (\ref{equality}),
$$
<\frac{{\alfa_2+n\choose n}}{\vol(K)}\frac{\alfa_2}{\alfa_1}t_0^{\alfa_2-\alfa_1}
\left[\int_0^{t_0} (V_f(t)- V_{\hat f}(t))\,d(t^{\alfa_1})-\int_{t_0}^{\hat M} (V_{\hat f}(t)- V_f(t))\,d(t^{\alfa_1})\right]=0
$$
so $\Phi_{\alfa_2}(f)<\Phi_{\alfa_2}(\hat f)=\hat M$, and since $\hat M$ has been chosen so that $\hat M=\Phi_{\alfa_1}( f)$, (\ref{decreasing2}) is proved.

If we have equality in \eqref{decreasing2}, we then have that $V_f(t)=V_{\hat f}(t)$. This means in particular that
\[
V_f(t)^{\frac1n}=V_{\hat f}(t)^\frac1n=\left(1-\frac{t}{\hat M}\right)\vol(K)^\frac1n,
\]
which by Brunn-Minkowski equality case implies that
\[
\{(x,t)\in K\times \R: 0\le t\le f(x)\}
=
\textrm{conv}(K\times \{0\},\{(x_1,\hat M)\})
\]
for some $x_1\in\R^n$.
\end{proof}

From (\ref{decreasing2}), we  derive Theorem \ref{berwald}, valid for  a finite number of positive concave functions.






We will use (\ref{alfabetanum}) to derive an extension of a theorem from \cite{jensen}. There, the $f_i$'s were supposed to be 1-variable integrable functions, and the normalization $\sum_{i=1}^m\alpha_i=1$ was also assumed.

\begin{proposition} Let $K\in\mathcal K^n$, $f_i:K\to[0,+\infty)$ be continuous and non identically null functions on $\mathrm{int}(K)$, $\alpha_1,\dots,\alpha_m>0$, and $\sigma={\alpha_1+\cdots +\alpha_m}$. Then
\begin{equation}\label{alfabetafunc}
\frac{1}{\vol(K)}\int_K \prod_{i=1}^m f_i(x)^{\alpha_i}\,dx
\le
\prod_{i=1}^m\left(
\frac{1}{\vol(K)}\int_K f_i(x)^{\sigma}\,dx
\right)^{\frac{\alpha_i}{\sigma}}
\end{equation}
Equality holds (for $m>1$) if and only if for any $x\in K$ and any $2\le i\le m$,
$$
f_i(x)=k_i f_1(x)
$$
for some positive constants $k_2,\dots, k_m$.
\end{proposition}

\begin{proof}
Replacing each $f_i$ by $\lambda_i f_i$, we may assume that
$$
\frac{1}{\vol(K)}\int_K f_i(x)^{\sigma}\,dx=1.
$$
Under this assumption, we have to prove that
$$
\frac{1}{\vol(K)}\int_K \prod_{i=1}^m f_i(x)^{\alpha_i}\,dx
\le1
$$
For any fixed $x\in K$, apply (\ref{alfabetanum}) with $\beta_i=f_i(x)$ to obtain
$$
f_1(x)^{\alpha_1}
\cdot
f_2(x)^{\alpha_2}
\cdots
f_m(x)^{\alpha_m}
\le
\sum_{i=1}^m\frac{\alpha_i}{\sigma}f_i(x)^{\sigma}
$$
Integrating over $K$ we get the desired inequality. Equality for $m>1$ stands if and only if all the normalized functions are the same, and from this, the condition follows.
\end{proof}

In the particular case of concave functions, we may use (\ref{decreasing2}) and (\ref{alfabetafunc}) to get

\begin{proof}[Proof of Theorem \ref{berwald}]
Let $\sigma={\alpha_1+\cdots +\alpha_m}$. For $1\le i\le m$, use (\ref{decreasing2}) with   $\alfa_1=\alpha_i$, $\alfa_2=\sigma$, $f=f_i$ to get
$$
\left(
\frac{{\sigma+n\choose n}}{\vol(K)}
\int_K f_i(x)^{\sigma}\,dx
\right)^{\frac{\alpha_i}{\sigma}}
\le
\frac{{\alpha_i+n\choose n}}{\vol(K)}\int_Kf_i(x)^{\alpha_i}\,dx,
$$
 Multiplying in $i=1,\dots m$
 $$
{ \sigma+n\choose n}
\prod_{i=1}^m
\left(
\frac{1}{\vol(K)}
\int_K f_i(x)^{\sigma}\,dx
\right)^{\frac{\alpha_i}{\sigma}}
\le
\prod_{i=1}^m\frac{{\alpha_i+n\choose n}}{\vol(K)}\int_Kf_i(x)^{\alpha_i}\,dx,
$$
 and using (\ref{alfabetafunc}), the result follows.

 Equality holds for $m>1$ in (\ref{decreasing2}) if and only if all $f_i$ are  roof functions, and in (\ref{alfabetafunc}) if and only if all $f_i$'s are proportional. So, all the $f_i$'s are roof functions over the same point in $K$.
\end{proof}
\section*{Acknowledgment}
Part of this work was carried out at the IMUS of the University of Sevilla and the authors are grateful for the hospitality and the support provided by the IMUS and MINECO grant MTM2015-63699-P during their stay.

\end{document}